\documentclass[11 pt, onecolumn, draftcls]{IEEEtran}

\usepackage{amsbsy,color,multicol}
\usepackage{floatflt} 

\usepackage{amsmath}
\usepackage{amssymb}
\usepackage{times}
\usepackage{graphicx}
\usepackage{xspace}
\usepackage{paralist} 
\usepackage{setspace} 
\usepackage{xypic}
\xyoption{curve}
\usepackage{latexsym}
\usepackage{amsthm}
\usepackage{ifthen}
\usepackage{caption}
\usepackage{subcaption}
\usepackage{makeidx,multirow,textcomp,longtable,afterpage,setspace}
\usepackage[english]{babel}
\usepackage[latin1]{inputenc}
\usepackage[ruled,vlined]{algorithm2e}
\usepackage{booktabs}
\usepackage{comment}

\usepackage{hyperref}

\newlength\myindent
\newlength\mycolwid

{
\newtheorem{theorem}{Theorem}
\newtheorem{lemma}[theorem]{Lemma}

\newtheorem{corollary}[theorem]{Corollary}

}

\title{\LARGE\bf Thermodynamic Limit of Interacting Particle Systems over Time-varying Sparse Random Networks}

\author{Augusto Almeida Santos$^\star$, Soummya Kar$^\star$, Jos\'e M. F. Moura$^\star$, Jo\~ao Xavier$\dagger$

\thanks{This work was partially supported by NSF grant $\#$ CCF-$1513936$,  Fundac\~ao para a Ci\^encia e a Tecnologia grant [UID/EEA/50009/2013] and part of the work was developed under the grant SFRH/BD/33516/2008 by the Funda\c{c}\~{a}o para a Ci\^{e}ncia e a Tecnologia through the Carnegie Mellon Portugal program.}
\thanks{$^\star$ A. A. Santos, J. M. F. Moura, S. Kar are with the Dep.~of Electrical and Computer Engineering,
Carnegie Mellon University, Pittsburgh, PA 15213, USA (augusto.pt@gmail.com, moura@ece.cmu.edu, soummyak@andrew.cmu.edu).}
\thanks{$\dagger$ J. Xavier is with the Institute for Systems and Robotics (ISR/IST), LARSyS, Instituto Superior T\'ecnico, Universidade de Lisboa, Av.~Rovisco Pais, Lisboa, Portugal (jxavier@isr.ist.utl.pt).}}

\begin{document}

\maketitle
\thispagestyle{empty}
\pagestyle{plain}

\begin{abstract}
We establish a functional weak law of large numbers for observable macroscopic state variables of interacting particle systems (e.g., voter and contact processes) over fast time-varying sparse random networks of interactions. We show that, as the number of agents~$N$ grows large, the proportion of agents~$\left(\overline{Y}_{k}^{N}(t)\right)$ at a certain state~$k$ converges in distribution -- or, more precisely, weakly with respect to the uniform topology on the space of \emph{c\`adl\`ag} sample paths -- to the solution of an ordinary differential equation over any compact interval~$\left[0,T\right]$. Although the limiting process is Markov, the prelimit processes, i.e., the normalized macrostate vector processes~$\left(\mathbf{\overline{Y}}^{N}(t)\right)=\left(\overline{Y}_{1}^{N}(t),\ldots,\overline{Y}_{K}^{N}(t)\right)$, are non-Markov as they are tied to the \emph{high-dimensional} microscopic state of the system, which precludes the direct application of standard arguments for establishing weak convergence. The techniques developed in the paper for establishing weak convergence might be of independent interest.
\end{abstract}


\section{Introduction}

Systems of interacting agents -- often abstracted as interacting particle systems -- can model many applications of large-scale networked agents from voting processes to diffusion of opinions or epidemics in large populations; examples include the Harris contact process, the voter process, and the Glauber-Ising model~\cite{liggett,liggett1999stochastic,baxter} in statistical mechanics. Due to the large-scale of such systems, it is often not feasible to follow the pathwise dynamics of the \emph{high-dimensional} microstate, i.e., of the collective state of all individuals in the population. As an alternative, one could attempt to study the evolution of the system by observing macroscopic quantities, that is, the state variables defined as global averages or \emph{low-resolution} functionals of the microstate. For instance, in epidemics, observing the binary state of each individual -- infected or not infected -- is prohibitive in a large population; instead, one often tracks the fraction of infected nodes in the population -- a macroscopic observable. However, and in light of the discussion in the introduction of~\cite{penrose2}, while the microstate of a system and the local rules of interaction completely determine its evolution (determinism principle), two systems at the same macrostate may evolve very differently. For instance, two distinct communities~$A$ and~$B$ may each have~$50\%$ of infected individuals, but community~$A$ may have a large number of contacts intertwining infected and healthy individuals -- a microscopic information -- as opposed to community~$B$ that may have a more clustered configuration. These microscopic configurations cause very different evolutions of the system at the macroscale: for example, the infected population will tend to increase at a much faster rate in community~$A$. Fig.~\ref{fig:nonmarkov2} illustrates this example.
\begin{figure} [hbt]
\begin{center}
\includegraphics[scale= 0.5]{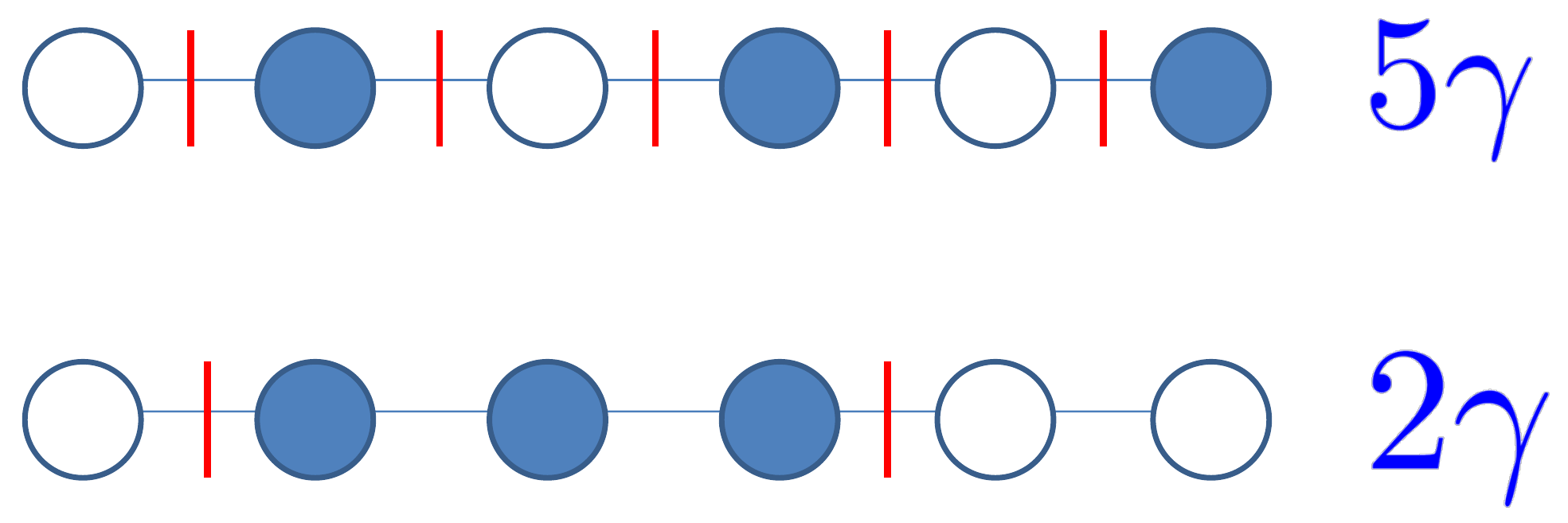}
\caption{Two line networks with the same number of infected nodes (blue/collored nodes). In the bottom network the configuration is clustered and the infection spreads only through the two links of contact between infected and healthy nodes. In the top network the infection can spread faster through the larger number of contact links between infected and healthy nodes.}\label{fig:nonmarkov2}
\end{center}
\end{figure}
It remains a major challenge in statistical mechanics to understand how the microscopics are exactly quotiented out, in the limit of large-scale interacting agents, to engender deterministic laws at the macroscale (a.k.a. fluid limit dynamics) without resorting to simplifying hypothesis of: i) full mixing: where the underlying network of interactions is complete\footnote{Any agent can interact with any other agent at any time.}; ii) ideal gases: no interaction among agents or iii) ergodicity. In such cases, the macrostates are overall Markov.

More formally, and for the sake of clarity, we introduce the concept of realization. Let~$\left(X(t)\right)$ be Markov. We say that the stochastic process~$\left(X(t)\right)$ is a refinement of the process~$\left(Y(t)\right)$ if~$\left(Y(t)\right)$ is measurable with respect to~$\left(X(t)\right)$. We say that~$\left(Y(t)\right)=\left(F\left(X(t)\right)\right)$ \emph{realizes} its refinement~$\left(X(t)\right)$ when~$\left(Y(t)\right)$ is Markov, i.e., its local (in time) evolution at each time~$t$ depends on the finer process~$\left(X(t)\right)$ only through~$\left(Y(t)\right)$ itself at each time~$t$. For instance, if~$F$ is bijective, then~$\left(Y(t)\right)$ trivially realizes~$\left(X(t)\right)$. In statistical mechanics,~$\left(X(t)\right)$ plays the role of the \emph{high-dimensional} microscopic process conveying all the information of the interacting particle system and~$\left(Y(t)\right)$ plays the role of the \emph{low-dimensional} macroscopic observable. In such framework,~$F$ is not bijective -- many different microstate configurations yield the same macroscopic observation.

In general, even though the microstate~$\left(X(t)\right)$ is Markov, the macrostates~$\left(Y(t)\right)$ are not as they are tied to the microstate, i.e., they do not realize the microstate. This paper proves that, in interacting particle systems over appropriate fast time-varying networks, the macrostates asymptotically realize the microstate and become Markov: as the number of agents grows large, knowledge of the macrostates~$Y(t)$ at present~$t$ becomes sufficient to foresee their immediate future. In other words, the determinism principle is recovered in the limit of large scale (time-varying) networks for the macrostate quantities. Formally, we prove that a sequence of non-Markov macrostate processes~$\left(\mathbf{\overline{Y}}^N(t)\right)$ converges in distribution -- i.e., weakly with respect to the uniform topology on the space of c\`adl\`ag sample paths -- to a process~$\left(\mathbf{y}(t)\right)$ that is the solution to an ordinary differential equation whose vector field only depends on~$\left(\mathbf{y}(t)\right)$, and thus, the limiting process~$\left(\mathbf{y}(t)\right)$ is Markov. In other words, the process~$\left(\overline{Y}^N(t)\right)$ realizes the microstate~$\left(X^N(t)\right)$ asymptotically in~$N$ (though not for finite~$N$).

When the exact ODE fluid limit dynamics associated with macroscopic state variables of a complex system exists, it provides a means to study interacting particle systems at the low-dimensional macroscale without keeping track of their microscopics; in particular, fluid limits are relevant to study the stability and metastability -- often observed in dynamical systems exhibiting multiple equilibria with \emph{large} basins of attraction -- of systems at the macroscale, e.g.,~\cite{mathieurobert,Antunes2,Queues} and, it is useful to study the qualitative dynamics of such large-scale systems outside their thermo-equilibrium (if there is one). These in turn maps to questions of determining thresholds and conditions on the parameters of the dynamical system under which this is ergodic or not (due to having multiple invariant measures). Hence it is worth seeking conditions on the underlying dynamics of the network (if not static) and of the interactions among agents that yield an exact fluid dynamics. To the best of our knowledge, such conditions are still sparse in the literature and exact concentration results for macroscopic variables associated with stochastic processes over networks are still a very open question in the literature: they are only available for complete networks, or trivial variants, e.g., supernetwork of densely connected cliques (a.k.a., network of communities), or complete-multipartite networks. The major reason lies in the fact that in such densely connected cases, the macroscopic variables of interest realize the microscopics, i.e., they are Markov and convergence results follow as corollary to now standard arguments such as Kurtz Theorem~\cite{Norris} (for weak convergence on the sample path), or Stein's method (for weak convergence of some limiting distribution, e.g., Gibbs measure), whereas for other settings, involving non-Markovian prelimit processes, asymptotic properties are relatively less explored. For instance, a complete network assumption is crucial to prove the limiting theorems in~\cite{Kirk} as it leads to full-exchangeability of the underlying macroscopic quantities (which in turn allows to resort to Stein's method). In our model we do not assume any of the densely connected type of networks mentioned and the underlying process is not exchangeable (though partial exchangeability is present as we will remark momentarily). In fact, in our case, the network will vary over time preserving its low-sparsity -- the number of edges is~$\sim O(N)$.
Note that one may also establish relevant bounds on macroscopic quantities -- instead of seeking exact fluid limits -- via stochastic domination on sparse networks by bounding the corresponding processes over complete network counter-parts as, e.g., in~\cite{Aldous22}. That is not the goal in the current paper. In this work, we show that for a particular dynamics of the network of contacts, namely, whenever there is an interaction between two agents, the agents randomly shuffle, then we obtain an exact fluid limit concentration. It lays still open, the question of determining the broadest class of network dynamics that leads to an exact fluid limit.

Note that one can obtain the limiting Partial Differential Equation dynamics of exclusion processes over lattices via the framework of hydrodynamics, e.g.,~\cite{Landim}: where one seeks to determine the PDE thermodynamic limit associated with the evolution of the number of particles locally in space (e.g., following exclusion processes dynamics). In particular, one is interested in the limiting behavior of a process~$\left(\eta^N(x,t)\right)$ -- the number of particles in a small interval or patch of space about~$x$, whose length is of order~$O(1/N)$, -- whereas we are interested in the evolution over time (without spatial dependence, i.e., ODE instead of a general PDE) of the fraction of agents at a particular state. The above framework is different from the one studied in this paper, e.g., the former requires a local renormalization of the scale of time and each vector-process~$\left(\eta^N(x,t)\right)$ -- where the vector entries collect the number of particles at each patch of the discretized space -- in the sequence is Markov.

To summarize, this paper shows that under a time-varying random rewiring dynamics of the \textbf{sparse} network of contacts or interactions (described in Section~\ref{sec:probform}), one can obtain \textbf{exact} weak convergence of the macrostate quantities, and, in particular, characterize their macroscale dynamics (ODE fluid limit). This work thus helps in filling the gap on the characterization of exact fluid limit dynamics on processes over large sparse random networks of contacts. As explained more formally in Section~\ref{sec:summarygoal}, besides classical interacting particle systems examples, the model adopted here may be applicable to other scenarios of practical interest, such as the case of malware spread in mobile communication networks, an application of increasing relevance given the emergence of large-scale distributed denial of service (DDoS) attacks~\cite{zhuolu,frommobi}.

\textbf{Outline of the paper.} Section~\ref{sec:probform} introduces the main definitions and the dynamics assumed; Section~\ref{sec:summarygoal} formalizes the main goal of the paper; Section~\ref{sec:convline} establishes an important result on the concentration of a rate process; Section~\ref{sec:frameswr} finally establishes the weak convergence of the macroprocess and illustrates a simulation result; and Section~\ref{sec:conclud} concludes the paper.

\section{Problem Formulation}\label{sec:probform}

%

In this section, we present the general class of interacting particle systems denoted as Finite Markov Information-Exchange (FMIE) introduced in~\cite{aldous2013}, and we define the model and dynamics we assume.

\subsection{Main Constructs}

Consider~$N$ agents and let~$X^N_{ik}(t)\in \left\{0,1\right\}$ be the indicator that node~$i$ is in state $k\in\mathcal{X}:=\left\{1,2,\ldots,K\right\}$ at time~$t$, where~$K<\infty$ is fixed. We represent by
\begin{equation}
\mathbf{X}^{N}(t)=\left(X_{ik}^N(t)\right)_{ik},\,\,X_{ik}^N(t)\in \left\{0,1\right\}\nonumber
\end{equation}
the matrix microstate collecting the state of each of the~$N$ agents at time~$t$. Each node can only be at one particular state at a time, i.e., the rows of~$\left(\mathbf{X}^{N}(t)\right)$ sum to~$1$. The underlying network of potential interactions, at time~$t$, is captured by the binary adjacency matrix on~$N$ nodes~$A^N(t)\in \left\{0,1\right\}^{N \times N}$. Let
\begin{equation}
\mathbf{X}_k^{N}(t)=\left(X_{1k}^N(t),\ldots,X_{Nk}^{N}(t)\right)\in \left\{0,1\right\}^{N}\nonumber
\end{equation}
be the $k$-column of the matrix~$\mathbf{X}^{N}(t)$.
We consider the macroscopic state variables
\begin{equation}
Y_k^N(t)=\sum_{i=1}^N X_{ik}^N(t)=\mathbf{1}^{\top}\mathbf{X}_k^{N}(t)\nonumber
\end{equation}
to be the number of agents at the state~$k\in\mathcal{X}$ at time~$t$ and its normalized counterpart
\begin{equation}
\overline{Y}_k^N(t)=\frac{1}{N}\mathbf{1}^{\top}\mathbf{X}_k^{N}(t)\nonumber
\end{equation}
to be the fraction of nodes at the state~$k\in\mathcal{X}$. Also, denote by
\begin{equation}
\mathbf{\overline{Y}}^N(t)=\left(\overline{Y}_1^N(t),\ldots,\overline{Y}_K^N(t)\right),\nonumber
\end{equation}
the vector process representing the empirical distribution of nodes across states~$\mathcal{X}$.

The microstate~$\left(\mathbf{X}^{N}(t)\right)$ is updated by two distinct processes: i) the peer-to-peer interactions given by the microscopic dynamics; and ii) the network random rewiring. Both are described in the next two subsections.

\subsection{Peer-to-peer Interaction Dynamics}\label{subsec:peertopeer}

We assume~$d_i(t)$ clocks at each node~$i$, where~$d_i(t)$ is the degree of node~$i$ at time~$t$ and each clock is dedicated to a current neighbor of~$i$: once a clock ticks, agent~$i$ interacts with the corresponding neighbor (in the current network topology or geometry). The clocks are independent and exponentially distributed -- hence,~$\left(\mathbf{X}^{N}(t)\right)$ is Markov. The state~$\mathbf{X}^N(t)$ is updated by these local interactions. If~$i$ interacts with~$j$ (as the clock of~$i$ pointing to~$j$ rings) at time~$t$, the state of~$i$ and~$j$ are updated as
\begin{equation}
\left(\mathbf{e}_{i}^{\top}\mathbf{X}^N(t),\mathbf{e}_j^{\top} \mathbf{X}^N(t)\right)=G\left(\mathbf{e}_{i}^{\top}\mathbf{X}^N(t_-),\mathbf{e}_j^{\top} \mathbf{X}^N(t_-)\right)\nonumber
\end{equation}
where~$\mathbf{e}_{\ell}$ is the canonical vector with~$1$ at the $\ell$th entry, and zero otherwise, and~$G\,:\,\mathcal{E}\times \mathcal{E}\rightarrow \mathcal{E} \times \mathcal{E}\nonumber$ is the update function with
\begin{equation}
\mathcal{E}:= \left\{\mathbf{e}_1,\mathbf{e}_2,\ldots,\mathbf{e}_K\right\}.\nonumber
\end{equation}
$G$ is a function that maps the state of the interacting nodes onto the new state, similar to as defined in~\cite{aldous2013}. For instance, if a node from state~$3$ interacts with a node in state~$5$ then, the new state for both nodes will be given by the tuple~$G\left(\mathbf{e}_3,\mathbf{e}_5\right)$.

From the peer-to-peer dynamics described, the updates in the microstate~$\left(\mathbf{X}^{N}(t)\right)$ change the macrostate $\left(\mathbf{\overline{Y}}^N(t)\right)$ according to the following Martingale problem~\cite{Ethier,Diffusion} pathwise dynamics
\begin{equation}
\overline{Y}_{k}^N(\omega,t)  =  \overline{Y}_{k}^N(\omega,0)+\overline{M}_k^N(\omega,t) + \frac{1}{N}\int_{0}^t \underbrace{\sum_{m\ell\in \mathcal{X}^2} \gamma_{m\ell} c_{m\ell}(k) \mathbf{X}_{m}^{N\, \top}(s_-) A^{N}(s_-)\mathbf{X}_{\ell}^{N}(s_-)}_{=:\mathcal{F}^N_k\left(\mathbf{X}^N(t_-),A^N(t_-)\right)} ds\label{eq:pathwise}
\end{equation}
where~$\overline{Y}_{k}^N(\omega,t)$ is the fraction of agents at the state~$k$ at time~$t$ for the realization~$\omega$,~$\left(\overline{M}_{1}^N(t),\ldots,\overline{M}_{K}^N(t)\right)$ is a normalized martingale process (refer to equation~\eqref{eq:martrepre}),~$\gamma_{m\ell}$ is the rate of the exponential clocks from nodes at the state~$m$ to contact nodes at the state~$\ell$, and~$c_{m\ell}(k)\in\left\{-2,-1,0,1,2\right\}$ gives the increment in the number of nodes at state~$k$ due to the interaction between nodes in states~$m$ and~$\ell$. For instance, if whenever node~$i$ in state~$1$ interacts with node~$j$ at state~$2$ causes both~$i$ and~$j$ to turn to state~$4$, then~$c_{12}(4)=2$. The terms~$c_{m\ell}(k)$ are uniquely determined from the given update function~$G$. Note also that the clock-rates~$\gamma_{m\ell}$ may or may not depend upon the states~$m$ and~$\ell$ of the interacting nodes. For instance, for the analysis of contact processes,~$\gamma_{m\ell}$ is often assumed independent of the states and represented simply as~$\gamma$ (rate of infection).

\textbf{Remark on~$c_{m\ell}(k)$.} Note that, if two nodes at state~$k$ interact, than the number of nodes in state~$k$ cannot be incremented as a result of this interaction (the two nodes interacting are already at this state). Hence, tensor~$\mathbf{c}$ is constrained to~$c_{kk}(k)\leq 0$ for all~$k$. This leads, on the other hand, to the fact that the hyper-cube~$\left[0,1\right]^K$ is invariant under the stochastic dynamics~\eqref{eq:pathwise}.

\subsection{Rewiring Network Dynamics}\label{subsec:rewiring}

Once an update on the microstate happens, the edges of the network are randomly rewired, i.e.,
\begin{equation}
A^N(t)=P^{\top} A^N(t_-) P,\nonumber
\end{equation}
where~$P\in \mathcal{P}_{\sf er}(N)$ is drawn uniformly randomly from the set of~$N\times N$ permutation matrices~$\mathcal{P}_{\sf er}(N)$ -- each time an update occurs.

An alternative representation to the random rewiring is the following:
\begin{equation}
\mathbf{X}_{m}^{N\,\top}(t_-)\left(P^{\top}A^{N}(t_-)P\right)\mathbf{X}_{\ell}^{N}(t_-)=\left(P\mathbf{X}_{m}^{N}(t_-)\right)^{\top}A^{N}(t_-)\left(P\mathbf{X}_{\ell}^{N}(t_-)\right),\nonumber
\end{equation}
in other words, we can consider equivalently that the network~$A^N$ is fixed and that the position of the nodes permute just after an update. This interpretation is assumed throughout the paper. And in fact, such partial exchangeability allows us to consider any network~$A^N$ with a fixed number of edges. For each~$N$, we consider a regular network with degree~$d^N$ and thus, the degree, in a sense, controls the sparsity of the network -- or, if we will, the real-time bandwith of nodes for peer-to-peer contact. Note that, given any even number~$N$ and an arbitrary degree~$d<N$, one can always draw a regular bipartite graph from it. In other words, and for the sake of our problem, we can assume that the graph is regular bipartite (which will be convenient momentarily).




\section{Summary and Goal of the Paper}\label{sec:summarygoal}

The model introduced in Section~\ref{sec:probform} arises from the fact that often agents \emph{wander around} as their state evolves, instead of being static. The macroscopic dynamical laws derived may be useful, e.g., to study large-scale interacting particle systems supported in sparse dynamical network environments. One particular application is malware propagation~\cite{zhuolu} in mobile devices networks. Mobile devices move around fast in an ad-hoc manner -- hence, the underlying network of contacts changes fast over time in an ad-hoc manner -- and their real-time bandwith for peer-to-peer communication is usually low~\cite{botfarina} -- hence, the geometry of the support network of interactions is sparse -- nevertheless, the massive amount of infected mobile devices corresponds to a large-scale (botnet) network that may launch, for instance, DDoS attacks~\cite{mobtaxonomy,ddosattackswired,ddosnews}, a modern threat over the Internet-of-Things (IoT) of increasing importance~\cite{zhuolu,frommobi} and that may display non-trivial metastable behavior~\cite{augustosecurity}.

One can partition the class of interacting particle systems into three broad categories:
\begin{itemize}
  \item \textbf{Fast network mixing dynamics:} the support network dynamics, i.e., the movement of the interacting agents runs at a much faster time-scale than the interactions among the agents themselves;
  \item \textbf{Static network:} interactions among agents run at a much faster rate than the network dynamics;
  \item \textbf{Mesoscale:} both dynamics run at comparable speeds.
\end{itemize}
The fluid limit paradigm varies greatly depending on the class under analysis. Our work sits primarily upon the first category: \emph{fast network mixing dynamics}.

To summarize, the microstate~$\left(\mathbf{X}^{N}(t)\right)$ is updated in two ways: i) peer-to-peer interactions given by the microscopic dynamics; and ii) network random rewiring just described. Fig.~\ref{fig:shufflingfig} summarizes the model for the case of a contact process and assuming that (the fixed)~$A^N$ is a cycle network.
\begin{figure} [hbt]
\begin{center}
\includegraphics[scale= 0.5]{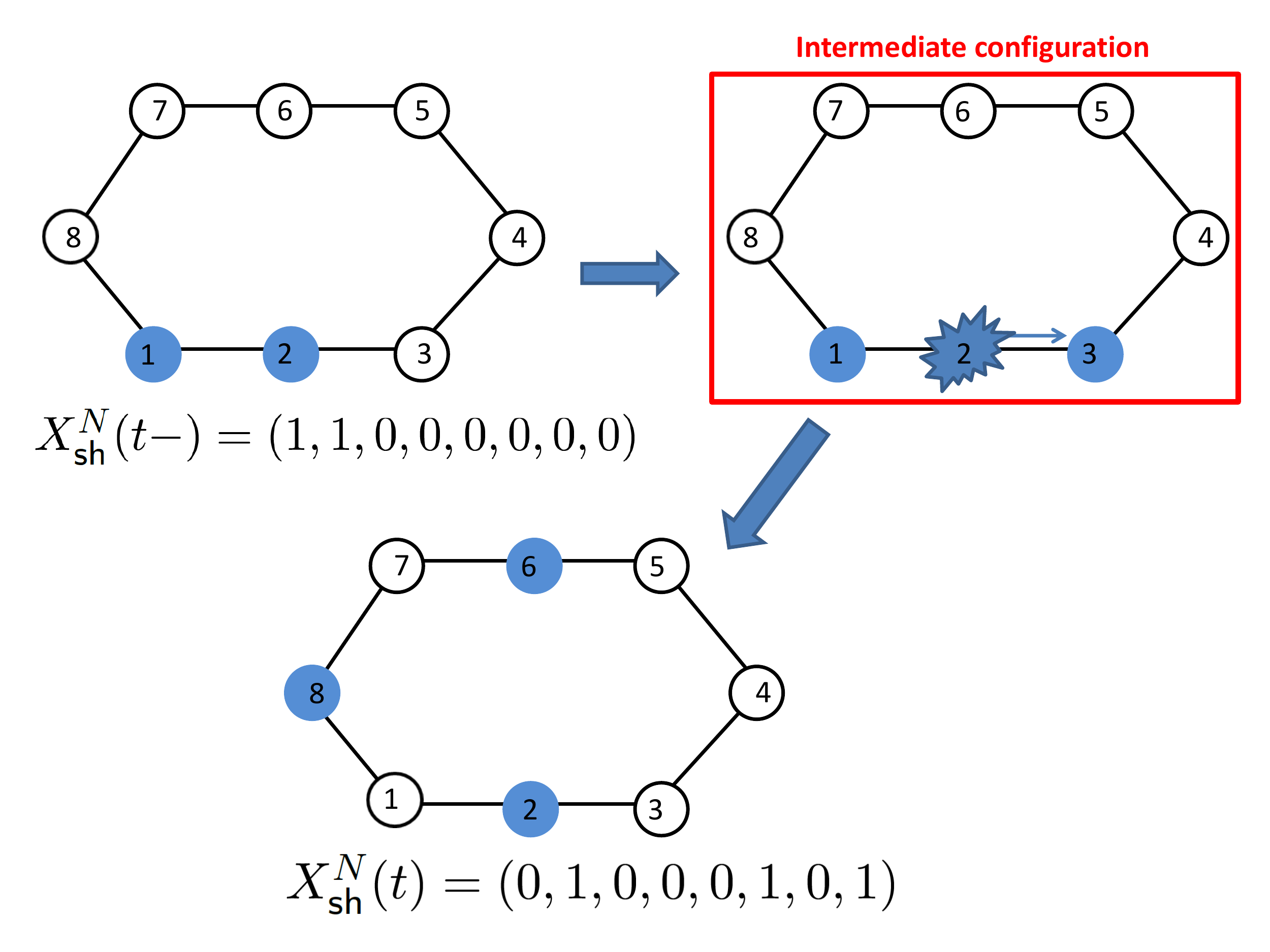}
\caption{Illustration of an update after an interaction. The clock of node~$2$ pointing to node~$3$ ticks. First, this leads to an update on the state of node~$3$ (it becomes blue); then a random permutation of the states of the nodes follow (or, equivalently, a rewiring of the edges of the graph happens -- this is not depicted in the figure). The figure illustrates the contact process over the random rewiring network.}\label{fig:shufflingfig}
\end{center}
\end{figure}


Note from equation~\eqref{eq:pathwise} that the empirical process~$\left(\mathbf{\overline{Y}}^N(t)\right)$ is not Markov as it is tied to the finer microscopics~$\left(\mathbf{X}^N(t)\right)$. Our goal is to prove that the non-Markovian sequence~$\left(\mathbf{\overline{Y}}^N(t)\right)$ converges weakly as~$N$ goes to infinite, with respect to the uniform topology on the space of c\`adl\`ag sample paths, to the solution~$\left(\mathbf{y}(t)\right)$ of the ODE
\begin{equation}
\dot{y}_{k}(t) = d \sum_{m\ell\in \mathcal{X}^2} \gamma_{m\ell} c_{m\ell}(k) y_m(t)y_{\ell}(t),\nonumber
\end{equation}
for~$k\in\mathcal{X}=\left\{1,\ldots,K\right\}$, where~$d$ is the asymptotic average degree of the limiting network, and~$\gamma_{m\ell}$ is the rate of the exponential clocks from nodes at state~$m$ to contact nodes at state~$\ell$. More compactly, the limiting equation is given by the ODE
\begin{equation}
\dot{y}_{k}(t) = d \,\,\mathbf{y}(t)^{\top}\left(\Gamma \odot C(k)\right) \mathbf{y}(t)=:f_k\left(\mathbf{y}(t)\right),\label{eq:odelower}
\end{equation}
where~$C(k):=\left(c_{m\ell}(k)\right)_{m\ell}$,~$\Gamma:=\left(\gamma_{m\ell}\right)_{m\ell}$, and~$\odot$ is the pointwise Hadamard product. Recall that the constraints on~$C$, in order to make the microscopic model physically meaningful (as referred at to the end of Subsection~\ref{subsec:peertopeer}) imply that the hyper-cube~$\left[0,1\right]^K$ is invariant under the ODE dynamics~\eqref{eq:odelower}.

This tacitly implies that, in order to perform macroscopic analysis in the large-scale, one can replace the complex stochastic microscopic dynamics in equation~\eqref{eq:pathwise} by the lower-dimensional ODE in equation~\eqref{eq:odelower}.

In the next sections, we prove weak convergence by establishing four major results:

\textbf{i) (Convergence on the line -- Subsection~\ref{subsec:largegap}).} The quadratic rate in the pathwise dynamics
\begin{equation}
\left|\frac{\mathbf{X}_m^{N\,\top}(t) A^N \mathbf{X}_{\ell}^N(t)}{N}-d^N \overline{Y}_m^N(t)\overline{Y}_{\ell}^N(t)\right| \overset{\mathbb{P}}\Longrightarrow 0\nonumber
\end{equation}
concentrates on the line (i.e., at each time~$t$) in probability exponentially fast. To prove this step, it is crucial to introduce and prove the result for an auxiliary process~$\left(\mathbf{\widetilde{X}}^N(t)\right)$ that is coupled with the original one~$\left(\mathbf{X}^N(t)\right)$, which is done in Subsection~\ref{subsec:largeaux}.

\textbf{ii) (Tightness -- Subsection~\ref{subsec:tightgap}).} Since, for each time~$t$, the quadratic rate converges exponentially fast, then it is tight via Theorem~\ref{th:poissondecay} -- note that simply convergence on the line does not imply tightness~\cite{billi}.

\textbf{iii) (Martingale converges to zero in probability -- Subsection~\ref{subsec:martconverges}).} We show that
\begin{equation}
\left|\left|M_k^N(t)\right|\right|_{\sup\left[0,T\right]}\overset{\mathbb{P}}\longrightarrow 0,\nonumber
\end{equation}
for every~$k\in\mathcal{X}$, where~$\left|\left|\cdot\right|\right|_{\sup\left[0,T\right]}$ is the sup-norm on the space of sample paths over the interval~$\left[0,T\right]$.

\textbf{iii) (Weak convergence -- Section~\ref{sec:frameswr}).} Relying on points i), ii) and iii), and by a standard evocation of the Skorokhod Representation Theorem~\cite{skorokhod,billi}, one can show weak convergence of the empirical process~$\left(\mathbf{\overline{Y}}^N(t)\right)$ to the solution of an ODE over each compact interval~$\left[0,T\right]$.

In what follows, we refer to the process
\begin{equation}
R^N_{m\ell}(t):=\frac{\mathbf{X}_m^{N\,\top}(t) A^N \mathbf{X}_{\ell}^N(t)}{N}-d^N \overline{Y}_m^N(t)\overline{Y}_{\ell}^N(t)\nonumber
\end{equation}
as the gap process.

\section{Weak Convergence of the Gap Process}\label{sec:convline}


In this section, we prove the following concentration in probability for the gap process (refer to Theorem~\ref{th:sup} for its formal statement)
\begin{equation}
\left|\left|\frac{\mathbf{X}_m^{N\,\top}(t) A^N \mathbf{X}_{\ell}^N(t)}{N}-d^N \overline{Y}_m^N(t)\overline{Y}_{\ell}^N(t)\right|\right|_{\sup\left[0,T\right]}\overset{\mathbb{P}}\Longrightarrow 0\nonumber
\end{equation}
where~$\overline{Y}_k^N(t)$ is the fraction of nodes at the state~$k\in\mathcal{X}$ at time~$t\in\left[0,T\right]$. In other words, the gap process is tight and this will be crucial to establish weak convergence of the macroscopic process~$\left(\mathbf{\overline{Y}}^N(t)\right)$. To prove such tightness result, we first establish it on the line for an auxiliary process that is coupled to our original microscopic process~$\mathbf{X}^N(t)$.

\subsection{Conditional Large Deviation on the Line for an Auxiliary Process}\label{subsec:largeaux}

Let~$\left(\mathbf{\widetilde{X}}_m^{N}(t)\right)=\left(\widetilde{X}_{1m}^{N}(t),\ldots,\widetilde{X}_{Nm}^{N}(t)\right)$ be defined as follows: whenever there is an interaction at~$t-$, and assuming we have~$\alpha_m N$ nodes at state~$m$ just after that interaction, the coordinates~$\widetilde{X}_{im}^{N}(t)$ are updated by the realization of~$N$ i.i.d. Bernoulli random variables (conditioned on~$\mathbf{\overline{Y}}^N(t)$) with conditional-law
\begin{equation}
\mathbb{P}\left(\widetilde{X}_{im}^{N}(t)=1\left|\overline{Y}_m^N(t_-)=\alpha_m\right.\right)=\alpha_m.\nonumber
\end{equation}

\textbf{Remark.} One can partition the set of edges~$E^N$ comprising~$A^N$ into sets of independent edges~$E^N=\bigcup_k E^N_k$ (a.k.a., matchings~\cite{west2008}), i.e., edges that do not share nodes in common. In other words,
\begin{equation}
i_1 i_2,\,i_3 i_4 \in E^N_k \Rightarrow i_m\neq i_n\mbox{ for any }m\neq n,\,m,n\in\left\{1,2,3,4\right\},\nonumber
\end{equation}
and note from the definition of~$\mathbf{\widetilde{X}}_m^N(t)$ that~$\widetilde{X}_{i_1m}(t)\widetilde{X}_{i_2m}(t)$ and~$\widetilde{X}_{i_3m}(t)\widetilde{X}_{i_4m}(t)$ are independent (given~$\mathbf{\overline{Y}}^N(t)$) if~$i_1 i_2,\,i_3 i_4 \in E_k$ for some~$k$. A matching that pairs all nodes is called a perfect matching. It is a rather simple to prove and well-known fact that the set of edges of a $d^N$-regular bipartite graph admits a partition into~$d^N$ perfect matchings of size~$N$ (number of nodes), e.g., refer to~\cite{bondy}.

The following lemma follows from a Bernstein concentration inequality~\cite{chung2006,concentration} for independent random variables (refer to Corollary~\ref{co:indepe} in the Appendix). For simplicity, in what follows, we denote~$\alpha_{m\ell}:=\alpha_m\alpha_{\ell}$. Let~$\left[\alpha\right]$ round~$\alpha>0$ so that~$\left[\alpha N\right]\in\mathbb{N}$ is the closest integer (from above) to~$\alpha N$.

\begin{lemma}\label{lem:concentra}
Let~$t\in\left[0,T\right]$ be fixed. For any~$\epsilon>0$, there is~$N_0,k>0$ so that
\begin{equation}
\mathbb{P}\left(\left|\frac{{\mathbf{\widetilde{X}}}_m^{N\,\top}(t)A^N {\mathbf{\widetilde{X}}_{\ell}}^{N}(t)}{N}-d \left[\alpha_{m\ell}\right]\right|>\epsilon \left|\overline{Y}_m^{N}(t_-)=\left[\alpha_{m}\right],\overline{Y}_{\ell}^{N}(t_-)=\left[\alpha_{\ell}\right]\right.\right)\leq 2e^{-kN},\nonumber
\end{equation}
for all~$N>N_0$, where~$k$ does not depend on~$\alpha_{m\ell}$ (this latter information is relevant for Theorem~\ref{th:explineconv} and follows from Corollary~\ref{co:indepe} in the Appendix).
\end{lemma}

\begin{proof}
As discussed in the Subsection~\ref{subsec:rewiring}, without loss of generality, we assume that~$A^N$ is a regular bipartite network with degree~$d^N$. One can thus partition the quadratic term~${\mathbf{\widetilde{X}}_m}^{N\,\top}(t)A^N {\mathbf{\widetilde{X}}_{\ell}}^{N}(t)$ into~$d^N$ sums of~$N$ independent terms
\begin{equation}
{\mathbf{\widetilde{X}}_m}^{N\,\top}(t)A^N {\mathbf{\widetilde{X}}_{\ell}}^{N}(t)= \sum_{i_1 j_1\in E_1} \widetilde{X}^N_{i_1 m} \widetilde{X}^N_{j_1 \ell} +\ldots + \sum_{i_d j_d \in E_d} \widetilde{X}^N_{i_dm} \widetilde{X}^N_{j_d\ell}\nonumber
\end{equation}
where each sum runs over a perfect matching and comprises~$N$ independent terms (as remarked before), where each term has mean~$\alpha_{m\ell}$. Thus,
\begin{eqnarray}
\mathbb{P}\left(\left|\frac{{\mathbf{\widetilde{X}}}_m^{N\,\top}(t)A^N {\mathbf{\widetilde{X}}_{\ell}}^{N}(t)}{N}-d^N \left[\alpha_{m\ell}\right]\right|>\epsilon \left|\overline{Y}_m^{N}(t_-)=\left[\alpha_{m}\right],\overline{Y}_{\ell}^{N}(t_-)=\left[\alpha_{\ell}\right]\right.\right) & \leq & \nonumber\\
\mathbb{P}\left(\left|\frac{\sum_{i_1 j_1\in E_1} \widetilde{X}^N_{i_1 m} \widetilde{X}^N_{j_1 \ell} +\ldots + \sum_{i_d j_d \in E_d} \widetilde{X}^N_{i_dm} \widetilde{X}^N_{j_d\ell}}{N}-d^N \left[\alpha_{m\ell}\right]\right|>\epsilon \left|\overline{Y}_m^{N}(t_-)=\left[\alpha_{m}\right],\overline{Y}_{\ell}^{N}(t_-)=\left[\alpha_{\ell}\right]\right.\right) & \leq & \nonumber\\
\mathbb{P}\left(\left|\frac{\sum_{i_1 j_1} \widetilde{X}^N_{i_1 m}\widetilde{X}^N_{j_1 \ell}}{N}-\left[\alpha_{m\ell}\right]\right| +\ldots + \left|\frac{\sum_{i_d j_d} \widetilde{X}^N_{i_d m}\widetilde{X}^N_{j_d \ell}}{N}-\left[\alpha_{m\ell}\right]\right|>\epsilon \left|\overline{Y}_m^{N}(t_-)=\left[\alpha_{m}\right],\overline{Y}_{\ell}^{N}(t_-)=\left[\alpha_{\ell}\right]\right.\right) & \leq & \nonumber \\
\mathbb{P}\left(\left|\frac{\sum_{i_1 j_1} \widetilde{X}^N_{i_1 m}\widetilde{X}^N_{j_1 \ell}}{N}-\left[\alpha_{m\ell}\right]\right|>\frac{\epsilon}{d^N}\left|\overline{Y}_m^{N}(t_-)=\left[\alpha_{m}\right],\overline{Y}_{\ell}^{N}(t_-)=\left[\alpha_{\ell}\right]\right.\right)+\ldots & &\nonumber\\
+\ldots + \mathbb{P}\left(\left|\frac{\sum_{i_d j_d}  \widetilde{X}^N_{i_d m}\widetilde{X}^N_{j_d \ell}}{N}-\left[\alpha_{m\ell}\right]\right|>\frac{\epsilon}{d^N} \left|\overline{Y}_m^{N}(t_-)=\left[\alpha_{m}\right],\overline{Y}_{\ell}^{N}(t_-)=\left[\alpha_{\ell}\right]\right.\right) & \leq & \nonumber \\
2d^N e^{-kN}\nonumber
\end{eqnarray}
for~$N$ large enough, where~$k$ is a function of the degree~$d^N$ and~$\epsilon$, but it does not depend on~$\alpha_{m\ell}$ (refer also to Corollary~\ref{co:indepe} in the Appendix); and the last inequality follows from the Bernstein concentration inequality.
\end{proof}

\subsection{Large Deviations on the Line for the Gap Process}\label{subsec:largegap}

Now, we observe that the main process~$\left(\mathbf{X}^{N}(t)\right)$ can be obtained (in distribution) from~$\left(\mathbf{\widetilde{X}}^{N}(t)\right)$ as follows: for any~$m\in\mathcal{X}$
\begin{enumerate}
  \item if~$\mathbf{1}^{\top}\mathbf{\widetilde{X}}_m^{N}(t) > \mathbf{1}^{\top}\mathbf{X}_m^{N}(t)$, then choose randomly~$\mathbf{1}^{\top}\mathbf{\widetilde{X}}_m^{N}(t)-\mathbf{1}^{\top}\mathbf{X}_m^{N}(t)$ of the~$1$'s of~$\left(\mathbf{\widetilde{X}}_m^{N}(t)\right)$ to flip to zero and declare the new vector as~$\mathbf{Z}_m^N(t)$;
  \item if~$\mathbf{1}^{\top}\mathbf{\widetilde{X}}_m^{N}(t) < \mathbf{1}^{\top}\mathbf{X}_m^{N}(t)$, then choose randomly $\mathbf{1}^{\top}\mathbf{X}_m^{N}(t)- \mathbf{1}^{\top}\mathbf{\widetilde{X}}_m^{N}(t)$ of the zero's of~$\left(\mathbf{\widetilde{X}}_m^{N}(t)\right)$ to flip to one and declare the new vector as~$\mathbf{Z}_m^N(t)$;
  \item if~$\mathbf{1}^{\top}\mathbf{\widetilde{X}}_m^{N}(t) = \mathbf{1}^{\top}\mathbf{X}_m^{N}(t)$, then set~$\mathbf{Z}_m^{N}(t)=\mathbf{\widetilde{X}}_m^{N}(t)$.
\end{enumerate}
Clearly,~$\mathbf{Z}^{N}(t)\overset{d}=\mathbf{X}^{N}(t)$ and the above construction couples both processes $\mathbf{X}^{N}(t)$ and $\mathbf{\widetilde{X}}^{N}(t)$ as we can write
\begin{equation}
\mathbf{\widetilde{X}}_m^{N}(t)+\mathbf{E}_m^N(t)\overset{d}=\mathbf{X}_m^{N}(t)\nonumber
\end{equation}
where the vector~$\mathbf{E}_m^N(t)\in \left\{-1,0,1\right\}^N$ flips the appropriate entries of the vector~$\mathbf{\widetilde{X}}_m^N(t)$, and the above equality holds in distribution for each~$m\in\left\{1,\ldots,K\right\}$. We have the theorem.

\begin{theorem}\label{th:coupling}
The following holds
\begin{equation}
\mathbb{P}\left(\left|\frac{\mathbf{X}_m^{N\,\top}(t)A^N \mathbf{X}_{\ell}^{N}(t)}{N}-\frac{{\mathbf{\widetilde{X}}_m}^{N\,\top}(t)A^N {\mathbf{\widetilde{X}}_{\ell}}^{N}(t)}{N}\right|> \epsilon \left|\mathbf{\overline{Y}}^{N}(t_-)\right.\right)\leq 2e^{-kN}.\nonumber
\end{equation}
where~$k$ does not depend on~$\mathbf{\overline{Y}}^N(t)$.
\end{theorem}

\begin{proof}
We get successively
\begin{eqnarray}
& \mathbb{P}\left(\left|\frac{\mathbf{X}_m^{N\,\top}(t)A \mathbf{X}_{\ell}^{N}(t)}{N}-\frac{\mathbf{\widetilde{X}}_{m}^{N\,\top}(t)A \mathbf{\widetilde{X}}_{\ell}^{N}(t)}{N}\right|> \epsilon \left|\mathbf{\overline{Y}}^{N}(t_-)=\mathbf{\alpha}\right.\right) & = \nonumber\\
& \mathbb{P}\left(\left|\frac{\left(\mathbf{\widetilde{X}}_{m}^{N\,\top}(t)+\mathbf{E}_m^N(t)\right)^{\top}A \left(\mathbf{\widetilde{X}}_{\ell}^{N\,\top}(t)+\mathbf{E}_{\ell}^N(t)\right)}{N}-\frac{\mathbf{\widetilde{X}}_{m}^{N\,\top}(t)A \mathbf{\widetilde{X}}_{\ell}^{N}(t)}{N}\right|> \epsilon \left|\mathbf{\overline{Y}}^{N}(t_-)=\mathbf{\alpha}\right.\right) & =\nonumber\\
& \mathbb{P}\left(\left|\frac{\mathbf{E}_{m}^{N\,\top}(t)A \mathbf{\widetilde{X}}_{\ell}^{N}(t)}{N}+\frac{\mathbf{\widetilde{X}}_m^{N\,\top}(t)A \mathbf{E}_{\ell}^{N}(t)}{N}+\frac{\mathbf{E}_m^{N\,\top}(t)A \mathbf{E}_{\ell}^{N}(t)}{N}\right|> \epsilon \left|\mathbf{\overline{Y}}^{N}(t_-)=\mathbf{\alpha}\right.\right) & \leq \nonumber\\
& \mathbb{P}\left(\left|\frac{\mathbf{E}_{m}^{N\,\top}(t)A \mathbf{\widetilde{X}}_{\ell}^{N}(t)}{N}\right| >\frac{\epsilon}{3} \left|\mathbf{\overline{Y}}^{N}(t-)=\mathbf{\alpha}\right.\right)+\mathbb{P}\left(\left|\frac{\mathbf{\widetilde{X}}_{m}^{N\,\top}(t)A \mathbf{E}_{\ell}^{N}(t)}{N}\right| >\frac{\epsilon}{3} \left|\mathbf{\overline{Y}}^{N}(t_-)=\mathbf{\alpha}\right.\right) & \nonumber\\
& + \mathbb{P}\left(\left|\frac{\mathbf{E}_{m}^{N\,\top}(t)A \mathbf{E}_{\ell}^{N}(t)}{N}\right| >\frac{\epsilon}{3} \left|\mathbf{\overline{Y}}^{N}(t_-)=\mathbf{\alpha}\right.\right) & \leq\nonumber\\
& \mathbb{P}\left(d^N \left|\frac{\mathbf{E}_{m}^{N\,\top}(t)\mathbf{1}}{N}\right| >\frac{\epsilon}{3} \left|\mathbf{\overline{Y}}^{N}(t_-)=\mathbf{\alpha}\right.\right)+\mathbb{P}\left(d^N \left|\frac{\mathbf{1}^{\top} \mathbf{E}_{\ell}^{N}(t)}{N}\right| >\frac{\epsilon}{3} \left|\mathbf{\overline{Y}}^{N}(t_-)=\mathbf{\alpha}\right.\right) & \nonumber\\
& + \mathbb{P}\left(d^N \left|\frac{\mathbf{E}_{m}^{N\,\top}(t)\mathbf{1}}{N}\right| >\frac{\epsilon}{3} \left|\mathbf{\overline{Y}}^{N}(t_-)=\mathbf{\alpha}\right.\right). & \nonumber
\end{eqnarray}
Each term on the right hand side of the last inequality can be bounded as follows
\begin{eqnarray}
& \mathbb{P}\left(d^N\left|\frac{\mathbf{E}_{i}^{N\,\top}(t)\mathbf{1}}{N}\right|> \frac{\epsilon}{3} \left|\mathbf{\overline{Y}}^{N}(t_-)=\mathbf{\alpha}\right.\right) & =\nonumber\\
& \mathbb{P}\left(\left|\frac{\mathbf{1}^{\top} \left(\mathbf{\widetilde{X}}_i^N(t)-\mathbf{X}_i^N(t)\right)}{N}\right|> \frac{\epsilon}{3 d^N} \left|\mathbf{\overline{Y}}^{N}(t_-)=\mathbf{\alpha}\right.\right) & =\nonumber\\
& \mathbb{P}\left(\left|\frac{ \left(\mathbf{1}^{\top}\mathbf{\widetilde{X}}_i^N(t)- \alpha_i N\right)}{N}\right|> \frac{\epsilon}{3 d^N} \left|\mathbf{\overline{Y}}^{N}(t_-)=\mathbf{\alpha}\right.\right) & =\nonumber\\
& \mathbb{P}\left(\left|\frac{\left(\mathbf{1}^{\top}\mathbf{\widetilde{X}}_i^N(t) \right)}{N}-\alpha_i\right|> \frac{\epsilon}{3 d^N} \left|\mathbf{\overline{Y}}^{N}(t_-)=\mathbf{\alpha}\right.\right) & \leq\nonumber\\
& 2 e^{-kN}\nonumber
\end{eqnarray}
for any~$\alpha$, where~$k$ does not depend on~$\alpha$. And the theorem follows.
\end{proof}

The next theorem follows as corollary to Lemma~\ref{lem:concentra} and Theorem~\ref{th:coupling}.

\begin{theorem}\label{th:explineconv}
We have
\begin{equation}
\mathbb{P}\left(\left|\frac{\mathbf{X}_m^{N\,\top}(t)A^N \mathbf{X}_{\ell}^{N}(t)}{N}-d^N \overline{Y}_m^{N}(t)\overline{Y}_{\ell}^{N}(t)\right|>\epsilon \right)\leq M e^{-kN},\nonumber
\end{equation}
for all~$t\geq 0$, and some~$M>0$.
\end{theorem}

\begin{proof}
\begin{eqnarray}
& \mathbb{P}\left(\left|\frac{\mathbf{X}_m^{N\,\top}(t)A^N \mathbf{X}_{\ell}^{N}(t)}{N}-d^N \overline{Y}_m^{N}(t)\overline{Y}_{\ell}^{N}(t)\right|>\epsilon \right) & = \nonumber\\
& E\left[\mathbb{P}\left(\left|\frac{\mathbf{X}_m^{N\,\top}(t)A^N \mathbf{X}_{\ell}^{N}(t)}{N}-d^N \overline{Y}_m^{N}(t)\overline{Y}_{\ell}^{N}(t)\right|>\epsilon \left| \mathbf{\overline{Y}}^N (t_-) \right.\right)\right] & =\nonumber\\
& E\left[\mathbb{P}\left(\left|\frac{\mathbf{X}_m^{N\,\top}(t)A^N \mathbf{X}_{\ell}^{N}(t)-\mathbf{\widetilde{X}}_m^{N\,\top}(t)A^N \mathbf{\widetilde{X}}_{\ell}^{N}(t)}{N}+\frac{\mathbf{\widetilde{X}}_{m}^{N\,\top}(t)A \mathbf{\widetilde{X}}_{\ell}^{N}(t)}{N}-d^N \overline{Y}_m^{N}(t)\overline{Y}_{\ell}^{N}(t)\right|>\epsilon \left| \mathbf{\overline{Y}}^N(t_-) \right.\right)\right] & \leq\nonumber\\
& E\left[\mathbb{P}\left(\left|\frac{\mathbf{X}_m^{N\,\top}(t)A \mathbf{X}_{\ell}^{N}(t)-\mathbf{\widetilde{X}}_{m}^{N\,\top}(t)A \mathbf{\widetilde{X}}_{\ell}^{N}(t)}{N}\right|>\epsilon \left| \mathbf{\overline{Y}}^N(t_-) \right.\right)\right] & \nonumber\\
& +E\left[\mathbb{P}\left(\left|\frac{\mathbf{\widetilde{X}}_m^{N\,\top}(t)A^N \mathbf{\widetilde{X}}_{\ell}^{N}(t)}{N}-d^N \overline{Y}_m^{N}(t)\overline{Y}_{\ell}^{N}(t)\right|>\epsilon \left| \mathbf{\overline{Y}}^N(t_-) \right.\right)\right] & \leq\nonumber\\
& M e^{-kN}, &\nonumber
\end{eqnarray}
where the last inequality follows from Lemma~\ref{lem:concentra}, Theorem~\ref{th:coupling}, and the fact that~$k$ does not depend on~$\left(\mathbf{\overline{Y}}^N(t_-)\right)$.
\end{proof}

%
%

\subsection{Tightness of the Gap Process}\label{subsec:tightgap}

The following theorem is crucial to what follows
\begin{theorem}\label{th:poissondecay}
Let~$M^{N}\overset{d}\sim \mathcal{N}_N\left(0,1\right)$ be a Poisson random variable with parameter~$N$. Let~$\left(Z^N_{i}\right)_{i=1}^{\infty}$ be a sequence of independent (and independent of~$M^{N}$) Bernoulli random variables with law
\begin{equation}
\mathbb{P}\left(Z^N_{i}=1\right)=\frac{1}{N^\alpha},\nonumber
\end{equation}
for all~$i\in\mathbb{N}$, with~$\alpha>1$. Then,
\begin{equation}
\sum_{i=0}^{M^N} Z^N_{i}\overset{\mathbb{P}}\longrightarrow 0,\nonumber
\end{equation}
or equivalently,
\begin{equation}
\mathbb{P}\left(\sum_{i=0}^{M^N} Z^N_{i}\geq 1\right)\longrightarrow 0.\nonumber
\end{equation}
as~$N$ goes to infinite.
\end{theorem}
The idea behind this theorem is that~$Z^N_{i}$ will play the role of the indicator of an $\epsilon$-deviation in our gap process
\begin{equation}
\left(R^N_{m\ell}(t)\right)=\left(\frac{\mathbf{X}_m^N(t)A^N\mathbf{X}_{\ell}^N(t)}{N}-\overline{Y}_{m}^N(t)\overline{Y}_{\ell}^N(t)\right).\nonumber
\end{equation}
Thus, the theorem states that the probability that there will be at least one $\epsilon$-deviation during the whole time interval~$\left[0,T\right]$ (i.e., across all shuffles in~$\left[0,T\right]$) decreases to zero as~$N$ grows large (as stated formally in Theorem~\ref{th:sup}).

\begin{proof}
First note that
\begin{eqnarray}
\mathbb{P}\left(\left.\sum_{i=0}^{M^N} Z^N_{i}\geq 1\right|M^N\right) & = & 1-\mathbb{P}\left(\left.\sum_{i=0}^{M^N} Z^N_{i} = 0\right|M^N\right)= 1-\mathbb{P}\left(\left.Z^N_{i} = 0\,\,\forall{i\leq M^N}\right|M^N\right)\nonumber\\
& = & 1-\left(1-\frac{1}{N^{\alpha}}\right)^{M^N} = 1-\left(\left(1-\frac{1}{N^{\alpha}}\right)^{N^{\alpha}}\right)^{M^N/N^{\alpha}}\nonumber\\
& = & 1-e(N)^{M^N/N^{\alpha}},\nonumber
\end{eqnarray}
where we defined
\begin{equation}
e(N):=\left(1-\frac{1}{N^{\alpha}}\right)^{N^{\alpha}}.\nonumber
\end{equation}
Now,
\begin{equation}
\mathbb{P}\left(\sum_{i=0}^{M^N} Z^N_{i}\geq 1\right)= E\left[\mathbb{P}\left(\left.\sum_{i=0}^{M^N} Z^N_{i}\geq 1\right|M^N\right)\right]=
\sum_{k\geq 0} \left(1-e(N)^{k/N^{\alpha}}\right) \frac{N^{k} e^{-N}}{k!}.\nonumber
\end{equation}
We have that
\begin{eqnarray}
e^{-N}\sum_{k}\left(1-e(N)^{k/N^{\alpha}}\right) \frac{N^{k}}{k!} & = & e^{-N}\left(\sum_{k}\frac{N^{k}}{k!} - \sum_{k}e(N)^{k/N^{\alpha}}\frac{N^{k}}{k!} \right)\nonumber\\
& = & 1-e^{-N}\sum_{k}e(N)^{k/N^{\alpha}}\frac{N^{k}}{k!}=1-e^{-N}\times e^{e(N)^{1/N^{\alpha}}N}\nonumber\\
& = & 1-e^{-N}\times e^{\left(1-\frac{1}{N^{\alpha}}\right)N}\nonumber\\
& = & 1- e^{-\frac{N}{N^{\alpha}}}\nonumber\\
& \overset{N\rightarrow \infty}\longrightarrow & 0,\nonumber
\end{eqnarray}
for any~$\alpha>1$.
\end{proof}

The next theorem is the main result of this section.

\begin{theorem}\label{th:sup}
We have
\begin{equation}
\lim_{N\rightarrow \infty}\mathbb{P}\left(\sup_{t\in\left[0,T\right]}\left|\frac{\mathbf{X}_m^{N\,\top}(t)A^N \mathbf{X}_{\ell}^{N}(t)}{N}-d^N \overline{Y}_{m}^{N}(t) \overline{Y}_{\ell}^{N}(t)\right|>\epsilon\right)=0,\nonumber
\end{equation}
for any~$\epsilon>0$.
\end{theorem}

\begin{proof}
Let~$M^N \sim \mathcal{N}_{d^N N}\left(0,T\right)$ be a Poisson random variable with parameter~$d^N N$ and let~$\widehat{M}^N$ count the number of interactions (i.e., a state change happens) across the time interval~$\left[0,T\right]$. Set
\begin{equation}
Z^N(t):=\mathbf{1}_{\left\{\left|\frac{\mathbf{X}_m^{N\,\top}(t)A^N \mathbf{X}_{\ell}^{N}(t)}{N}-d^N \overline{Y}_{m}^{N}(t) \overline{Y}_{\ell}^{N}(t)\right|>\epsilon\right\}}(t)\nonumber
\end{equation}
to be the indicator of an $\epsilon$-deviation in the gap process. Now, note that under an appropriate coupling
\begin{equation}
\widehat{M}^N:=\mbox{Number of actual shuffles during} \left[0,T\right] \leq_{a.s.} M^N \overset{d}\sim \mathcal{N}_{\lambda(N)}\left(0,T\right)\nonumber
\end{equation}
where~$\lambda(N)=d^N N$, and~$d^N$ is the degree of the network~$A^N$. In particular, the intensity~$\lambda$ of the Poisson upper-bounding the number of shuffles on the interval increases linearly with~$N$. It follows that
\begin{eqnarray}
\mathbb{P}\left(\sup_{t\in\left[0,T\right]}\left|\frac{\mathbf{X}_m^{N\,\top}(t)A^N \mathbf{X}_{\ell}^{N}(t)}{N}-d^N \overline{Y}_{m}^{N}(t) \overline{Y}_{\ell}^{N}(t)\right|>\epsilon\right) & = & \mathbb{P}\left(\sum_{i=1}^{\widehat{M}^N}Z^N(t_i)\geq 1\right) \nonumber\\
& \leq & \mathbb{P}\left(\sum_{i=0}^{M^N} Z^N_{i}\geq 1\right) \overset{N\rightarrow \infty}\longrightarrow 0\nonumber
\end{eqnarray}
where the last inequality follows from Theorem~\ref{th:poissondecay} and the large deviation on the line, Theorem~\ref{th:explineconv}.
\end{proof}


%
%
%

\subsection{Martingale Converges in Probability to Zero}\label{subsec:martconverges}

In this subsection, we prove the following theorem.
\begin{theorem}\label{lm:martconv}
For any~$\epsilon>0$, the following holds
\begin{equation}
\mathbb{P}\left(\sup_{\left[0,T\right]}\left|\overline{M}_k^N(t)\right|>\epsilon\right)\overset{N\rightarrow \infty}\longrightarrow 0\nonumber
\end{equation}
for each~$k\in\mathcal{X}$ and~$T\geq 0$.
\end{theorem}

\begin{proof}
We prove that for each~$T\geq 0$, we have
\begin{equation}
{\sf E}\left(\overline{M}_k^N(T)\right)^2 \overset{N\rightarrow \infty}\longrightarrow 0,
\end{equation}
that is, the martingale vanishes in~$\mathcal{L}^2$ on the line. The theorem will follow as corollary to Doob's inequality, i.e.,
\begin{equation}
P\left(\sup_{0\leq t\leq T}\left|\overline{M}_k^{N}(t)\right|>\epsilon\right)\leq \frac{{\sf E}\left(\overline{M}_k^N(T)\right)^2}{\epsilon^2}\overset{N\rightarrow \infty}\longrightarrow 0,\,\,\,\forall\epsilon>0,\,\,\,\forall T\geq0.\label{eq:Doobsinequality}
\end{equation}
For each~$k\in\mathcal{X}$, the martingale is given by
\begin{equation}
M^N_k(t)=\sum_{m\ell\in\mathcal{X}^2}\sum_{n=0}^{d^N N}\int_{0}^t c_{m\ell}(k) \mathbf{1}_{\left\{X_m^{N\top}(s_{-})A^N X_{\ell}^{N}(s_{-})=n \right\}}\left(\mathcal{N}^{(m,\ell,n)}_{\gamma_{m\ell}n}(ds)-\gamma_{m\ell} n \,ds\right)\label{eq:martrepre}
\end{equation}
where~$\left\{\mathcal{N}^{(m,\ell,n)}_{\gamma_{m\ell}n}\right\}_{(m,\ell,n)}$ is a family of pair-wise independent Poisson processes indexed by the triple~$(m,\ell,n)$ and each with mean or parameter~$\gamma_{m\ell}n$.
We have
\begin{eqnarray}
{\sf E}\left(M_{k}^{N}(T)\right)^2&=& {\sf E}\left(\sum_{m\ell\in\mathcal{X}^2}\sum_{n}\int_{0}^T c_{m\ell}(k) \mathbf{1}_{\left\{X_m^{N\top}(s_{-})A^N X_{\ell}^{N}(s_{-})=n \right\}}\left(\mathcal{N}^{(m,\ell,n)}_{\gamma_{m\ell}n}(ds)-\gamma_{m\ell} n\, ds\right)\right)^2\\
&=&\sum_{m\ell\in\mathcal{X}^2}\sum_{n}{\sf E}\left(\int_{0}^T c_{m\ell}(k) \mathbf{1}_{\left\{X_m^{N\top}(s_{-})A^N X_{\ell}^{N}(s_{-})=n \right\}}\left(\mathcal{N}^{(m,\ell,n)}_{\gamma_{m\ell}n}(ds)-\gamma_{m\ell} n\, ds\right)\right)^2\label{eq:mart2}\\
&=&\sum_{m\ell\in\mathcal{X}^2}\sum_{n}{\sf E}\left(\int_{0}^T c^2_{m\ell}(k) \mathbf{1}_{\left\{X_m^{N\top}(s_{-})A^N X_{\ell}^{N}(s_{-})=n \right\}} \gamma_{m\ell} n ds\right)\label{eq:mart3}\\
&\leq& \sum_{m\ell\in\mathcal{X}^2} {\sf E}\left(\int_{0}^{T}\sum_{n}  \mathbf{1}_{\left\{X_m^{N\top}(s_{-})A^N X_{\ell}^{N}(s_{-})=n \right\}} 4 \gamma d^N N ds\right)\label{eq:mart4}\\
&\leq& 4 K^2 \gamma d^N N T,\label{eq:mart5}
\end{eqnarray}
where~$\gamma:=\max_{m\ell} \gamma_{m\ell}$; the second equality~\eqref{eq:mart2} follows from Theorem~\ref{th:ortho} and the independence of all the underlying Poisson processes involved (hence, the cross terms in the square are zero-mean martingales). The third equality~\eqref{eq:mart3} is due to the It\^{o} isometry Theorem (refer to~\cite{Diffusion2} or~\cite{Karatzas}) and the fact that the quadratic variation of a compensated Poisson martingale is given by
\begin{equation}
\left\langle\mathcal{N}_{\gamma}(t)-\gamma t \right\rangle=\gamma t.\nonumber
\end{equation}
The first inequality~\eqref{eq:mart4} is due to
\begin{equation}
c^2_{m\ell}(k)\leq 4;\,\,\, n\leq  d^N N = 2 \times \# \mbox{ of edges}.\nonumber
\end{equation}
The last inequality~\eqref{eq:mart5} holds since the family of subsets of the interval~$\left[0,T\right]$
\begin{equation}
I_n(\omega):=\left\{s\in\left[0,T\right]\,:\, X_m^{N\top}(\omega,s_{-})A^N X_{\ell}^{N}(\omega,s_{-})=n\right\},\nonumber
\end{equation}
for each fixed pair~$\left(m,\ell\right)$, indexed by~$n$, are realization-wise disjoint and thus for each pair~$\left(m,\ell\right)$
\begin{equation}
\sum_n \mathbf{1}_{\left\{X_m^{N\top}(\omega,s_{-})A^N X_{\ell}^{N}(\omega,s_{-})=n\right\}}=
{\bf 1}_{\bigcup_{n}{\left\{X_m^{N\top}(\omega,s_{-})A^N X_{\ell}^{N}(\omega,s_{-})=n\right\}}}\leq {\bf 1}_{\left[0,T\right]}(\omega,s_{-}).\nonumber
\end{equation}
for all $\omega\in\Omega$.

Therefore, for the normalized martingale, we have for all fixed~$T$
\begin{equation}
{\sf E}\left(\overline{M}_{k}^{N}(T)\right)^2=\frac{1}{N^2}{\sf E}\left(M_{k}^{N}(T)\right)^2\leq \frac{4K^2 \gamma d^N T}{N}\longrightarrow 0.\nonumber
\end{equation}
and the result now follows from Doob's inequality~\eqref{eq:Doobsinequality}.
\end{proof}

\section{Weak Convergence of the Macroprocess}\label{sec:frameswr}

%


The stochastic dynamical system for the macroscopics~\eqref{eq:pathwise} can be rewritten as follows
\begin{eqnarray}
\overline{Y}_k^N(t) & = & \overline{Y}_{k}^N(0)+\overline{M}_k^N(t) + d^N\int_{0}^t \sum_{m\ell\in \mathcal{X}^2} \gamma_{m\ell} c_{m\ell}(k) \overline{Y}_m^N(s_-)\overline{Y}_{\ell}^N(s_-) ds\nonumber\\
& &\, + \int_{0}^t \sum_{m\ell\in \mathcal{X}^2} \gamma_{m\ell} c_{m\ell}(k) \left(\mathbf{X}_{m}^{N\, \top}(s_-) A^{N}\mathbf{X}_{\ell}^{N}(s_-)- d^N \overline{Y}_m^N(s_-)\overline{Y}_{\ell}^N(s_-)\right)ds.\nonumber
\end{eqnarray}
for each~$k\in\mathcal{X}$. The next theorem follows from the equicontinuity condition in the Arzel\`{a}-Ascoli Theorem (refer to Theorem~\ref{th:arzela} in the Appendix).

\begin{theorem}\label{th:finaltight}
The sequence of macro-processes~$\left(\mathbf{\overline{Y}}^N(t)\right)=\left(\overline{Y}_1^N(t),\ldots,\overline{Y}_K^N(t)\right)$ is $C$-tight, i.e., its set of weak-accumulation points is nonempty and lie almost surely in~$C_{\left[0,T\right]}$, that is,
\begin{equation}
\left\{\left(\mathbf{\overline{Y}}^{N_k}(t)\right)\Rightarrow \left(\mathbf{\overline{Y}}(t)\right)\right\}\Rightarrow \mathbb{P}\left(\left(\mathbf{\overline{Y}}(t)\right)\in C_{\left[0,T\right]}\right)=1.\nonumber
\end{equation}
\end{theorem}

\begin{proof}
We show that $\left(\mathbf{\overline{Y}}^{N}(t)\right)$ fulfills the bound and equicontinuity conditions in equations~\eqref{eq:arzela}-\eqref{eq:arzela2} in the Arzel\`{a}-Ascoli Theorem, Theorem~\ref{th:arzela}.
Indeed, we have
\begin{equation}
\mathbb{P}\left(\sup_{0\leq t\leq T}\overline{Y}_m^{N}(t)\geq k\right)=0,\, \forall k>1,\nonumber
\end{equation}
and the first condition holds since~$0 \leq \overline{Y}_m^{N}(t)\leq 1$ almost surely for all~$t\in\left[0,T\right]$ and~$m\in\mathcal{X}$ (as referred in the end of Subsection~\ref{subsec:peertopeer}).

For the equicontinuity condition, for every~$k\in\mathcal{X}$, we have
\begin{eqnarray}
\label{eq:variation}
\omega\left(\overline{Y}_k^{N}, \delta, T\right)&=&\sup_{\left|u-v\right|\leq \delta,\,u,v\in\left[0,T\right]}\left\{\left|\overline{Y}_k^{N}(u)-\overline{Y}_k^{N}(v)\right|\right\}\label{eq:variation1}\\
&=& \sup_{\left|u-v\right|\leq \delta,\,u,v\in\left[0,T\right]}\left\{\left|\overline{M}_k^{N}(u)-\overline{M}_k^{N}(v)\right.\right.\label{eq:variation2}\\
&&\left.\left.+d^N\int_{u}^{v} \sum_{m\ell\in \mathcal{X}^2} \gamma_{m\ell} c_{m\ell}(k) \overline{Y}_m^{N}(s_-)\overline{Y}_{\ell}^{N}(s_-) ds\right.\right.\nonumber\\
& & +\left.\left.\int_{u}^v \sum_{m\ell\in \mathcal{X}^2} \gamma_{m\ell} c_{m\ell}(k) \left(\frac{\mathbf{X}_m^{N\,\top}(s_-)A^N\mathbf{X}_{\ell}^N(s_-)}{N}-d^N \overline{Y}_m^{N}(s_-)\overline{Y}_{\ell}^{N}(s_-)\right) ds\right|\right\}\nonumber\\
&\leq & \sup_{0\leq t\leq T}\left|\overline{M}_k^{N}(t)\right|+\gamma^{(k)} d^N \delta+\gamma^{(k)} \delta  \sup_{0\leq t\leq T} \left|R_{m\ell}^N(t)\right|\\
&=:&\omega_2\left(\overline{Y}_k^{N},\delta,T\right),\label{eq:variation5}
\end{eqnarray}
where we defined~$\gamma^{(k)}:=\sum_{m\ell\in \mathcal{X}^2} \gamma_{m\ell} \left|c_{m\ell}(k)\right|$. Now, for any $\widehat{\epsilon}>0$, we have
\begin{equation}
\mathbb{P}\left(\omega\left(\overline{Y}_i^{N},\delta,T\right)\geq \widehat{\epsilon}\right)\leq \mathbb{P}\left(\omega_2\left(\overline{Y}_i^{N},\delta,T\right)\geq\widehat{\epsilon}\right).\nonumber
\end{equation}
Moreover,
\begin{eqnarray}
\mathbb{P}\left(\omega_2\left(\overline{Y}_k^{N},\delta,T\right)\geq\widehat{\epsilon}\right) & \leq &
\mathbb{P}\left(\sup_{0\leq t\leq T}\left|\overline{M}_k^{N}(t)\right|>\frac{\widehat{\epsilon}}{3}\right)+\mathbb{P}\left(\gamma^{(k)} d^N \delta >\frac{\widehat{\epsilon}}{3}\right)\label{eq:bothsides1}\\
& & +\mathbb{P}\left(\gamma^{(k)} \delta \sup_{0\leq t\leq T} \left|R_{m\ell}^N(t)\right| >\frac{\widehat{\epsilon}}{3}\right).\label{eq:bothsides2}
\end{eqnarray}
By applying the $\limsup_{N}$ on both sides of the inequality~\eqref{eq:bothsides1}-\eqref{eq:bothsides2}, we obtain
\begin{equation}
\limsup_{N\rightarrow \infty}\mathbb{P}\left(\omega_2\left(\overline{Y}_k^{N},\delta,T\right)\geq\epsilon\right)\leq
\mathbb{P}\left(\gamma d \delta >\frac{\widehat{\epsilon}}{3}\right),\label{eq:bothsides3}
\end{equation}
from Theorem~\ref{th:sup}, the martingale convergence Theorem~\ref{lm:martconv}, and the assumption $d^N\overset{N\rightarrow \infty}\longrightarrow d$. Therefore, we can apply $\lim_{\delta\rightarrow 0}$ to equation~\eqref{eq:bothsides3}
\begin{equation}
\lim_{\delta\rightarrow 0}\limsup_{N\rightarrow \infty}\mathbb{P}\left(\omega_2\left(\overline{Y}^{N},\delta,T\right)\geq\epsilon\right)\leq
\lim_{\delta\rightarrow 0}\mathbb{P}\left(\gamma d \delta>\frac{\widehat{\epsilon}}{3}\right)=0,\nonumber
\end{equation}
and thus,
\begin{equation}
\lim_{\delta\rightarrow 0}\limsup_{N\rightarrow \infty}\mathbb{P}\left(\omega\left(\overline{Y}_k^{N},\delta,T\right)\geq\epsilon\right)\leq \lim_{\delta\rightarrow 0}\limsup_{N\rightarrow \infty}\mathbb{P}\left(\omega_2\left(\overline{Y}_k^{N},\delta,T\right)\geq\epsilon\right)=0.\nonumber
\end{equation}
\end{proof}

We conclude that $\left(\mathbf{\overline{Y}}^{N}(t)\right)$ is a tight family with almost surely continuous weak-accumulation points,
\begin{equation}
\left(\mathbf{\overline{Y}}^{N_n}(t)\right)\overset{n\rightarrow \infty}\Rightarrow \left(\mathbf{\overline{Y}}(t)\right)\nonumber
\end{equation}
with
\begin{equation}
\mathbb{P}\left(\left(\mathbf{\overline{Y}}(t)\right)\in C_{\left[0,T\right]}\right)=1.\nonumber
\end{equation}

\begin{theorem}\label{eq:finalconvergence}
Let~$\mathbf{\overline{Y}}^{N}(0)\Rightarrow \mathbf{\overline{Y}}(0)$. Any weak accumulation process~$\left(\mathbf{\overline{Y}}(t)\right)$ of $\left(\mathbf{\overline{Y}}^N(t)\right)$ obeys the integral equation
\begin{equation}
\overline{Y}_k(\omega,t)=\overline{Y}_k(\omega,0)+d \sum_{m\ell\in \mathcal{X}^2} \int_{0}^{t} \gamma_{m\ell} c_{m\ell}(k) \overline{Y}_m(\omega,s)\overline{Y}_{\ell}(\omega,s)ds,\label{eq:integr}
\end{equation}
for~$k\in\mathcal{X}$ and almost all~$\omega\in \Omega$.
\end{theorem}

\begin{proof}
Define the functional
\begin{equation}
\mathcal{F}_k\,:\,D^{K\times K}_{\left[0,T\right]}\times D^K_{\left[0,T\right]}\longrightarrow \mathbb{R}\nonumber
\end{equation}
with
\begin{eqnarray}
\mathcal{F}_k\left(\left(\mathbf{r}(t),\mathbf{y}(t)\right)\right) & := & y_k(t)-y_k(0)+d \sum_{m\ell\in \mathcal{X}^2} \int_{0}^{t} \gamma_{m\ell} c_{m\ell}(k) y_m(s) y_{\ell}(s)ds\\
& & +\sum_{m\ell\in \mathcal{X}^2} \int_{0}^{t} \gamma_{m\ell} c_{m\ell}(k) r_{m\ell}(s)ds.\nonumber
\end{eqnarray}
where~$D^K_{\left[0,T\right]}$ stands for the space of c\`{a}dl\`{a}g sample paths from the interval~$\left[0,T\right]$ to the cube~$\left[0,1\right]^K$ endowed with the Skorokhod metric (it is a Polish space, refer to~\cite{billi}). The functional~$\mathcal{F}_k$ is measurable: indeed, the sum `$+$' operator is measurable (with respect to the product topology $D_{\left[0,T\right]}\times D_{\left[0,T\right]}$); the integral operator `$\left(\int_{0}^t (\cdot)ds\right)$' is measurable; and composition of measurable operators is measurable (for these observations, refer to~\cite{bogachev}).

Let~$\left(\mathbf{\overline{Y}}^{N_n}(t)\right)\Rightarrow \left(\mathbf{\overline{Y}}(t)\right)$ and remark from Theorem~\ref{th:sup} that~$\left(\mathbf{\overline{R}}^{N_n}(t)\right)\Rightarrow \mathbf{0}$. We now prove that
\begin{equation}
\mathcal{F}_k\left(\mathbf{\overline{R}}^{N_n}(t),\mathbf{\overline{Y}}^{N_n}(t)\right)\Rightarrow \mathcal{F}_k\left(\mathbf{0},\mathbf{\overline{Y}}(t)\right).\nonumber
\end{equation}
From the Skorokhod's Representation Theorem~\cite{skorokhod,billi},
\begin{equation}
\begin{array}{cc}
\exists \left(\mathbf{\widetilde{Y}}^{n}(t)\right),\,\left(\mathbf{\widetilde{R}}^{n}(t)\right),\, \left(\mathbf{\widetilde{Y}}(t)\right)\,: &  \left(\mathbf{\widetilde{Y}}^{n}(t)\right)\overset{d}= \left(\mathbf{\overline{Y}}^{N_n}(t)\right),\, \left(\mathbf{\widetilde{R}}^{n}(t)\right)\overset{d}= \left(\mathbf{\overline{R}}^{N_n}(t)\right)\\
& \left(\mathbf{\widetilde{Y}}(t)\right)\overset{d}=\left(\mathbf{\overline{Y}}(t)\right)\\
& \left(\mathbf{\widetilde{Y}}^{n}(\omega,t)\right)\overset{U\left[0,T\right]}\longrightarrow \left(\mathbf{\widetilde{Y}}(\omega,t)\right),\,\left(\mathbf{\widetilde{R}}^{n}(\omega,t)\right)\overset{U\left[0,T\right]}\longrightarrow \mathbf{0}.\nonumber
\end{array}
\end{equation}
for almost all $\omega\in \Omega$, where~$U\left[0,T\right]$ stands for uniform convergence in the interval~$\left[0,T\right]$. Since,
\begin{equation}
\left(\mathbf{\widetilde{Y}}^{n}(\omega,t)\right)\longrightarrow\left(\mathbf{\widetilde{Y}}(\omega,t)\right)\nonumber
\end{equation}
a.s. uniformly over the compact interval~$\left[0,T\right]$, we can interchange the limit with the integral via the Dominated Convergence Theorem (e.g.,~\cite{williamsprobability}),
\begin{eqnarray}
\int_{0}^{t}\widetilde{Y}_{m}^{n}(\omega,s)\widetilde{Y}_{\ell}^{n}(\omega,s)ds & \longrightarrow & \int_{0}^t \widetilde{Y}_{m}(\omega,s) \widetilde{Y}_{\ell}(\omega,s) ds\nonumber \\
\int_{0}^{t}\widetilde{R}_{m\ell}^{n}(\omega,s) ds & \longrightarrow & 0.\nonumber
\end{eqnarray}
Therefore,
\begin{equation}
\mathcal{F}_k\left(\mathbf{\overline{R}}^{N_n}(t),\mathbf{\overline{Y}}^{N_n}(t)\right)\overset{d}= \mathcal{F}_k\left(\mathbf{\widetilde{R}}^{n}(t),\mathbf{\widetilde{Y}}^{n}(t)\right)\longrightarrow \mathcal{F}_k\left(\mathbf{0},\mathbf{\widetilde{Y}}(t)\right)\overset{d}= \mathcal{F}_k\left(\mathbf{0},\mathbf{\overline{Y}}(t)\right)\label{eq:alterprocess}
\end{equation}
where the first and last equality are due to the measurability of $\mathcal{F}_k$; and the convergence `$\longrightarrow$' is in a realization-wise sense with respect to the uniform topology on the space of sample paths. In particular, this implies convergence in probability, and thus, convergence~\eqref{eq:alterprocess} holds in a weak sense (refer to Corollary~$1.6$ from~\cite{Ethier}), i.e.,
\begin{equation}
\mathcal{F}_k\left(\mathbf{\overline{R}}^{N_n}(t),\mathbf{\overline{Y}}^{N_n}(t)\right)\Rightarrow \mathcal{F}_k\left(0,\mathbf{\overline{Y}}(t)\right).\nonumber
\end{equation}
Remark that $\left(\mathbf{\overline{Y}}^{N}(t)\right)$ obeys the following stochastic dynamics
\begin{equation}
\mathcal{F}_k\left(\mathbf{\overline{R}}^{N}(\omega,t),\mathbf{\overline{Y}}^{N}(\omega,t)\right)=\overline{M}_k^{N}(\omega,t),\nonumber
\end{equation}
and since
\begin{equation}
\left(\mathbf{\overline{M}}^{N}(t)\right)\Rightarrow 0\nonumber
\end{equation}
we have
\begin{equation}
\mathcal{F}_k\left(\mathbf{0},\mathbf{\overline{Y}}(\omega,t)\right)\overset{d}= 0,\nonumber
\end{equation}
or in other words,
\begin{equation}
\overline{Y}_k(\omega,t)=\overline{Y}_k(\omega,0)+d \sum_{m\ell\in \mathcal{X}^2} \int_{0}^{t} \gamma_{m\ell} c_{m\ell}(k) \overline{Y}_m(\omega,s)\overline{Y}_{\ell}(\omega,s)ds,\nonumber
\end{equation}
for almost all $\omega\in\Omega$.
\end{proof}

Now, from the uniqueness of the integral equation (the vector field is Lipschitz) we conclude uniqueness of the accumulation point and the following result follows.

\begin{theorem}
Let~$\mathbf{\overline{Y}}^N(0)\Rightarrow \mathbf{y}(0)$. We have
\begin{equation}
\left(\mathbf{\overline{Y}}^N(t)\right)\Rightarrow \left(\mathbf{y}(t)\right)=\left(y_1(t),\ldots,y_k(t)\right)\nonumber
\end{equation}
where~$\left(\mathbf{y}(t)\right)$ is the solution to the ODE
\begin{equation}
\dot{y}_{k}(t) = d \,\,\mathbf{y}(t)^{\top}\left(\Gamma \odot C(k)\right) \mathbf{y}(t)\mbox{  for }k=\left\{1,2,\ldots,K\right\}\label{eq:mainode}
\end{equation}
with initial condition~$\mathbf{y}(0)$, $\Gamma=\left[\gamma_{m\ell}\right]_{m\ell}$, $C(k)=\left[c_{m\ell}(k)\right]_{m\ell}$ and $\odot$ is the pointwise Hadamard product.
\end{theorem}

\begin{proof}
Since the vector field is Lipschitz, the continuous (and thus, differentiable) solution $\left(\mathbf{\overline{Y}}(t)\right)$ of (\ref{eq:integr}) is unique. Thus, any weak limit of $\left(\mathbf{\overline{Y}}^N(t)\right)$ with initial condition given by $\mathbf{\overline{Y}}^N(0)$ and converging in distribution to $\mathbf{\overline{Y}}(0)$ is equal to the unique solution $\left(\mathbf{\overline{Y}}(t)\right)$ of~(\ref{eq:integr}) with initial condition $\left(\mathbf{\overline{Y}}(0)\right)$. Therefore, by Prokhorov's Theorem~\cite{billi,bogachev}, the whole sequence converges
\begin{equation}
\left(\mathbf{\overline{Y}}^N(t)\right)\Rightarrow\left(\mathbf{\overline{Y}}(t)\right)\nonumber
\end{equation}
to the solution of~(\ref{eq:integr}). Equation~(\ref{eq:integr}) is the integral version of the ODE~\eqref{eq:mainode}.
\end{proof}

Fig.~\ref{fig:concentrationillu} depicts a numerical simulation illustrating the concentration result proved for the case of a binary state~$\left\{0,1\right\}$ contact process, where~$A^N$ was assumed to be a cycle network (i.e.,~$d^N=2$) for all~$N=100,1000,4000$. We observe that as the number of nodes increase, the stochastic dynamics (captured by the blue noisy curves) concentrates about the solution to the limiting ODE (captured by the red smooth curves).
\begin{figure}[ht!]
\begin{center}
\begin{subfigure}{0.31\textwidth}
\centering
\includegraphics[scale=0.4]{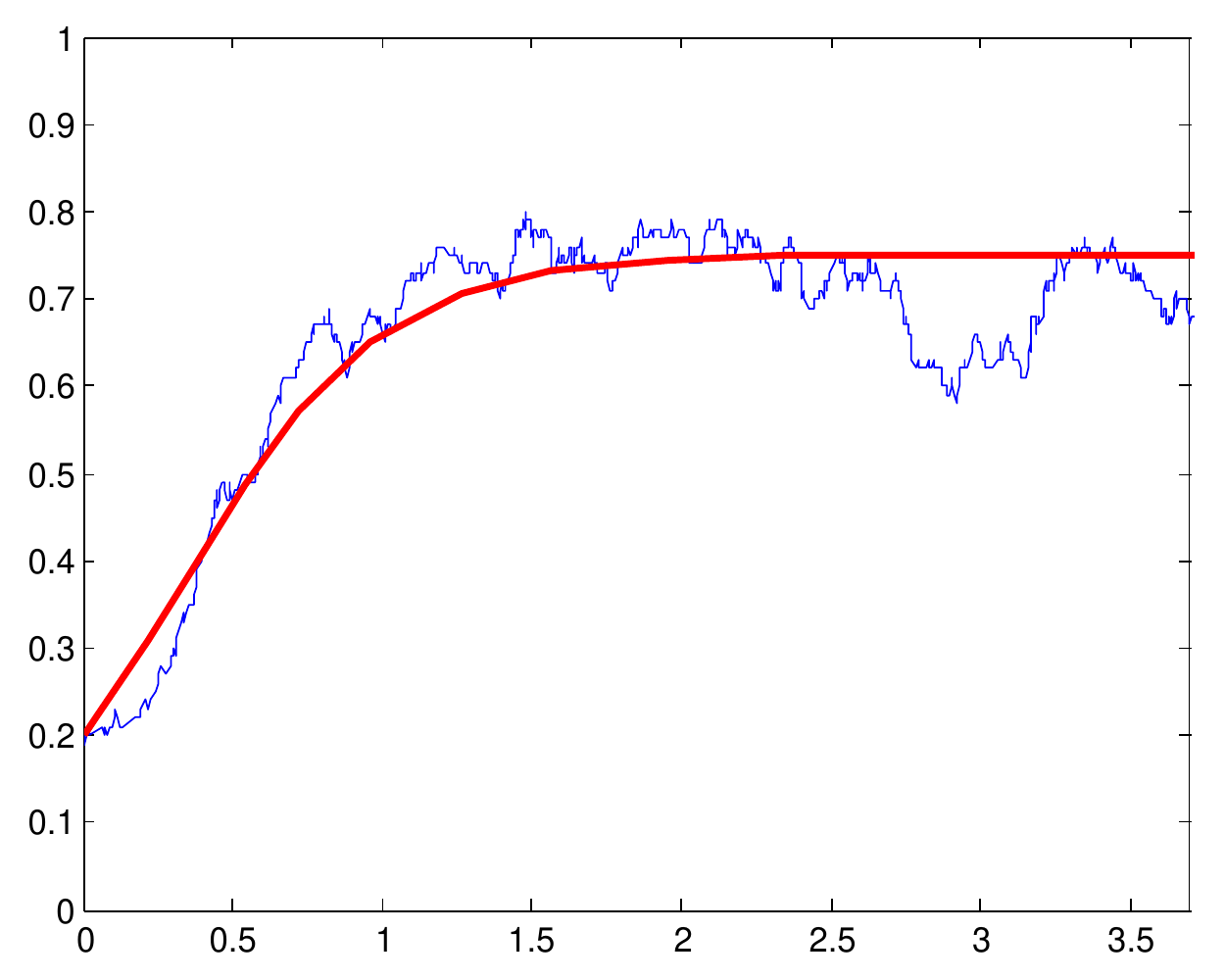}
\caption{$100$ nodes.} \label{fig:hittings}
\end{subfigure}
\begin{subfigure}{0.31\textwidth}
\centering
\includegraphics[scale=0.4]{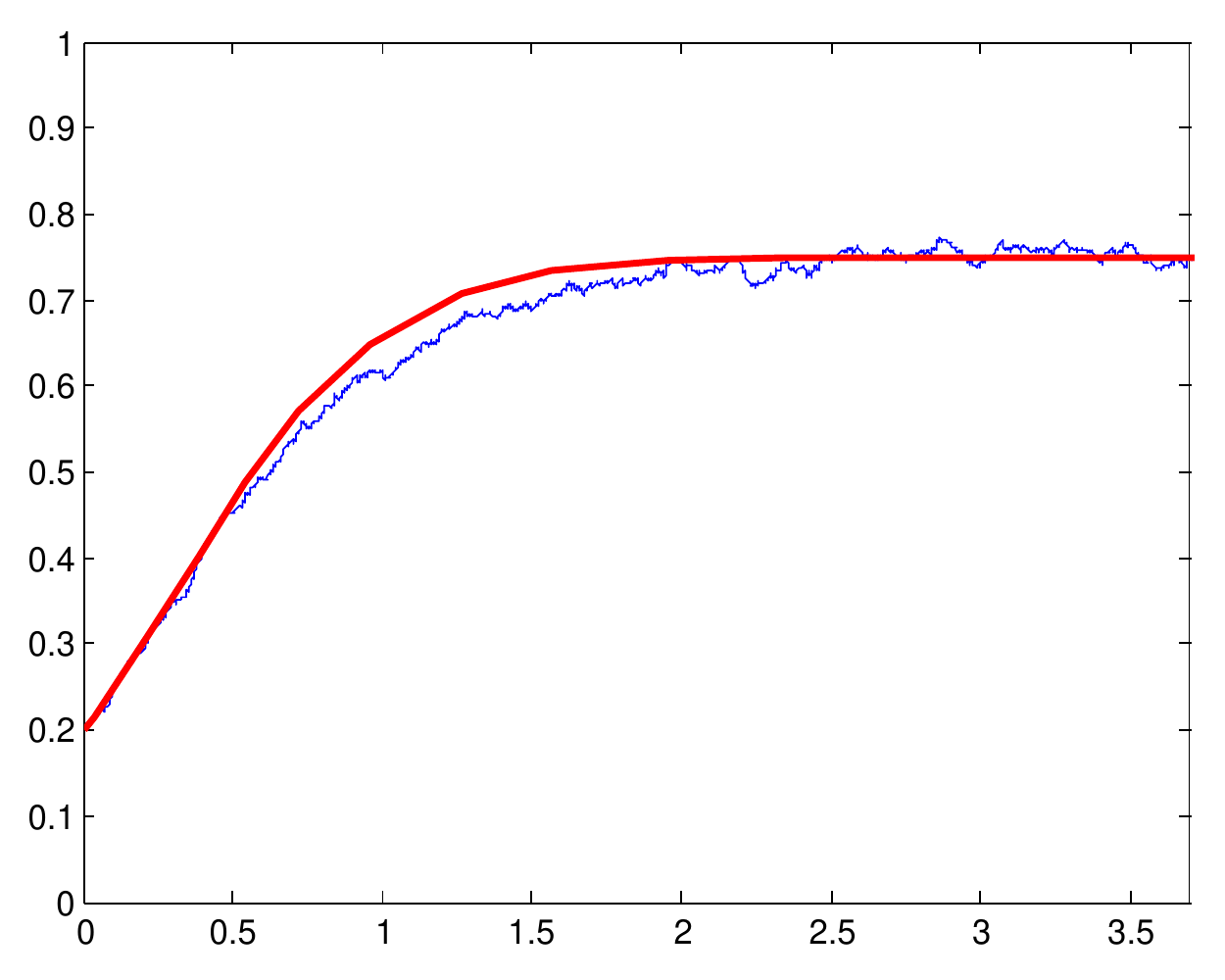}
\caption{$1000$ nodes.} \label{fig:hittings2}
\end{subfigure}
\begin{subfigure}{0.31\textwidth}
\centering
\includegraphics[scale=0.4]{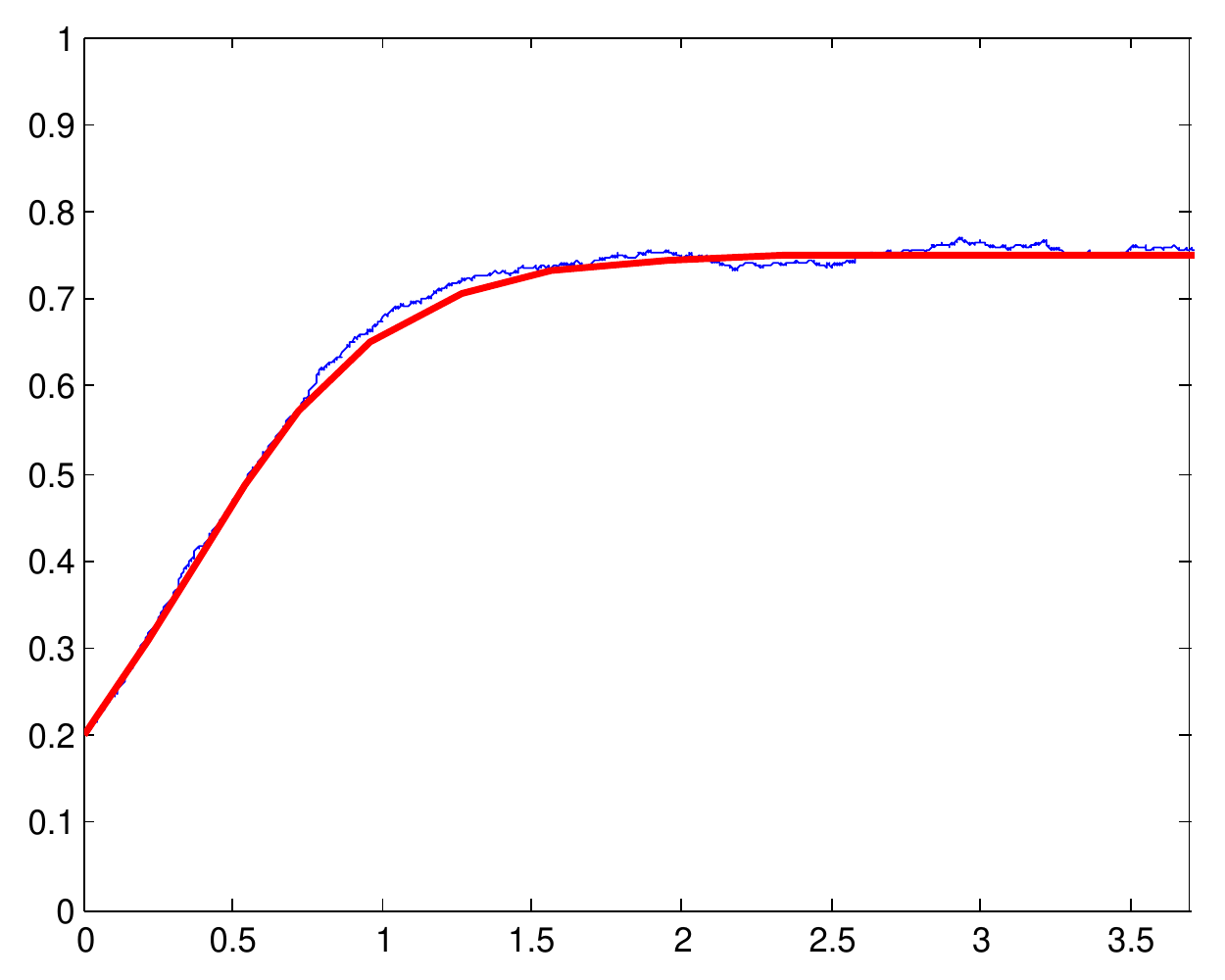}
\caption{$4000$ nodes.} \label{fig:hittings3}
\end{subfigure}
\caption{Evolution over time of the fraction of nodes at state~$1$ (a.k.a. infected). The blue noisy curve illustrates the original stochastic system and the smooth red curve illustrates the solution of the corresponding limiting differential equation.} \label{fig:concentrationillu}
\end{center}
\end{figure}

\section{Concluding Remarks}\label{sec:conclud}

Deriving the exact macroscopic dynamical laws from the microscopic laws of interacting particle systems is challenging when outside the scope of: uncoupled systems (a.k.a., ideal gases), full network of contacts, or network of communities. Within such frameworks, low-resolution macroscopic state variables, such as the fraction of nodes at a particular state, \emph{realize the system}.
In this paper, we proved that under a time-varying random rewiring dynamics of the sparse network of contacts (described in Subsection~\ref{subsec:rewiring}), the non-Markov macro-state variables associated with the fraction of nodes at each state~$k$ realize the system asymptotically in~$N$. That is, one can obtain the \textbf{exact} fluid limit macroscopic dynamics associated with general FMIE interacting particle systems. To establish such result, one has to primarily prove the tightness and finite-dimensional distribution convergence of built-in rate processes (e.g., the gap process converges to zero on the line) of the macroscopic process (e.g., the fraction of nodes at a particular state). The main difficulty in establishing such result for interacting particle systems over networks -- or general systems whose rules are set at the microscopics and respect the peer-to-peer disposition of nodes -- is that the pre-limit macroscopic processes are non-Markov (unless the underlying network of contacts is complete or it is a network of communities), and one of two steps is often hard: i) tightness of the rates; ii) convergence on the line of the rates. By introducing an intermediate process~$\left(\mathbf{\widetilde{X}}^N(t)\right)$, appropriately coupled with the original process~$\left(\mathbf{X}^N(t)\right)$, we were able to address both steps mentioned above. A natural future direction is to characterize more general classes of dynamical networks for which such exact concentration results are attainable.


\appendix~\label{ap:appendixee}

The next theorem provides with an important concentration inequality~\cite{chung2006,concentration}.
\begin{theorem}[Bernstein]\label{th:berny}
Let~$\left(Z_{i}\right)$ be a sequence of zero-mean independent random variables bounded by some constant~$c>0$, i.e.,~$\left|Z_i\right|\leq c$ a.s. for all~$i$. Let
\begin{equation}
\sigma^2(N)= \frac{1}{N} \sum_{i=1}^{N} Var\left(Z_{i}\right)\nonumber
\end{equation}
be the sample mean variance. Then, for any~$\epsilon>0$,
\begin{equation}
\mathbb{P}\left(\frac{1}{N}\sum_{i=1}^N Z_i\geq \epsilon\right)\leq e^{-\frac{N\epsilon^2}{2\sigma(N)^2+2 c \epsilon/3}}\nonumber
\end{equation}
\end{theorem}
We restate the previous theorem into a more useful corollary, as follows.
\begin{corollary}
Under the same assumptions as in Theorem~\ref{th:berny}, we have
\begin{equation}
\mathbb{P}\left(\left|\frac{1}{N}\sum_{i=1}^N Z_i \right|\geq \epsilon\right)\leq 2 e^{-\frac{N\epsilon^2}{2\sigma(N)^2+2 c \epsilon/3}}\nonumber
\end{equation}
\end{corollary}

\begin{proof}
Note that
\begin{equation}
\mathbb{P}\left(-\frac{1}{N}\sum_{i=1}^N Z_i\leq -\epsilon\right)\leq e^{-\frac{N\epsilon^2}{2\sigma(N)^2+2 c \epsilon/3}}\nonumber
\end{equation}
and by symmetry in the assumptions of the theorem -- namely, if~$\left(Z_i\right)_i$ fulfills the conditions, then~$\left(-Z_i\right)_i$ fulfills as well -- we have
\begin{equation}
\mathbb{P}\left(\frac{1}{N}\sum_{i=1}^N Z_i\leq -\epsilon\right)\leq e^{-\frac{N\epsilon^2}{2\sigma(N)^2+2 c \epsilon/3}}\nonumber
\end{equation}
and therefore,
\begin{eqnarray}
\mathbb{P}\left(\left|\frac{1}{N}\sum_{i=1}^N Z_i\right| \geq \epsilon \right) & = & \mathbb{P}\left(\left\{\frac{1}{N}\sum_{i=1}^N Z_i \geq \epsilon\right\} \cup \left\{\frac{1}{N}\sum_{i=1}^N Z_i\leq -\epsilon\right\}\right)\nonumber\\
& \leq & \mathbb{P}\left(\left\{\frac{1}{N}\sum_{i=1}^N Z_i \geq \epsilon\right\}\right) + \mathbb{P}\left( \left\{\frac{1}{N}\sum_{i=1}^N Z_i\leq -\epsilon\right\}\right)\nonumber\\
& \leq & 2 e^{-\frac{N\epsilon^2}{2\sigma(N)^2+2 c \epsilon/3}}\nonumber
\end{eqnarray}
\end{proof}

\begin{corollary}\label{co:indepe}
If in addition to the assumptions in Theorem~\ref{th:berny}, we have bounded variance, i.e.,
\begin{equation}
{\sf Var}\left(Z_i\right)\leq v,\,\,\forall{i\in\mathbb{N}}\nonumber
\end{equation}
for some~$v>0$, then
\begin{equation}
\mathbb{P}\left(\left|\frac{1}{N}\sum_{i=1}^N Z_i \right|\geq \epsilon\right)\leq 2 e^{-kN}\nonumber
\end{equation}
with
\begin{equation}
k=-\frac{\epsilon^2}{2v^2+2 c \epsilon/3}.\nonumber
\end{equation}
\end{corollary}

\begin{theorem}[Arzel\`a-Ascoli;\cite{Ruth}]\label{th:arzela}
Let $\left(\overline{Z}^{N}(t)\right)$ be a sequence of c\`{a}dl\`{a}g processes. Then, the sequence of probability measures $\mathbb{P}_{\overline{Z}^N}$ induced on $D_{\left[0,T\right]}$ by $\left(\overline{Z}^{N}(t)\right)$ is tight and any weak limit point of this sequence is concentrated on the subset of continuous functions $C_{\left[0,T\right]}\subset D_{\left[0,T\right]}$ if and only if the following two conditions hold for each~$\epsilon>0$:
\begin{eqnarray}
\lim_{k\rightarrow\infty}\limsup_{N\rightarrow\infty}\mathbb{P}\left(\sup_{0\leq t\leq T} \left|\overline{Z}^{N}(t)\right|\geq k\right) &= &0\mbox{         (Uniform Boundness)}\label{eq:arzela}\\
\lim_{\delta\rightarrow 0}\limsup_{N\rightarrow\infty}\mathbb{P}\left(\omega(\overline{Z}^{N},\delta,T)\geq \epsilon\right)&= &0\mbox{   (Equicontinuity)}\label{eq:arzela2}
\end{eqnarray}
where we defined the modulus of continuity
\begin{equation}
\omega(x, \delta, T)=\sup\left\{\left|x(u)-x(v)\right|\,:\,0\leq u,v\leq T,\, \left|u-v\right|\leq \delta\right\}.\nonumber
\end{equation}
\end{theorem}

\begin{theorem}[Orthogonality; Proposition A.10 in~\cite{Queues}]\label{th:ortho}
Let~$\left(\mathbf{Y}(t)\right)$ be an $\left(\mathcal{F}_{t}\right)$-adapted c\`adl\`ag process with discrete range and piecewise constant (i.e., constant when it does not jump). Let~$\mathcal{N}_{\lambda}(t)$ and~$\mathcal{N}_{\mu}(t)$ be two independent $\left(\mathcal{F}_{t}\right)$-adapted Poisson processes (hence their compensated versions are $\left(\mathcal{F}_{t}\right)$-martingales, as it is trivial to establish). Assume the rates~$\lambda,\mu$ are nonnegative. Let $f, g$ be two bounded functions defined over the discrete range of $\mathbf{Y}(t)$. Then,
\begin{equation}
\left(\int_{0}^t f\left(\mathbf{Y}(s-)\right)\left(\mathcal{N}_{\lambda}(ds)-\lambda ds\right) \int_{0}^t g\left(\mathbf{Y}(s-)\right)\left(\mathcal{N}_{\mu}(ds)-\mu ds\right)\right)
\end{equation}
is an $\left(\mathcal{F}_{t}\right)$-martingale.
\end{theorem}


\small
\bibliographystyle{IEEEtran}
\bibliography{IEEEabrv,biblio}

\end{document}